\documentclass[11pt]{amsart}

\usepackage{amsmath, amsthm, graphics,setspace,xypic,tikz,hyperref}
\usepackage{latexsym}
\usepackage{mathtools}
\usepackage{eucal}
\usepackage{caption}

\usepackage{amssymb}
\usepackage{color}
\usepackage{amscd}
\usepackage{palatino}
\usepackage{latexsym}
\usepackage{epsfig}
\usepackage{graphicx}
\usepackage{amsfonts}
\usepackage{psfrag}
\usepackage[all]{xy}
\usepackage{youngtab, ytableau}
\usepackage{url}
\usepackage{cite}

\usepackage[margin=1.3in]{geometry}

\definecolor{ltgrey}{RGB}{180, 187, 198}

\input xy
\xyoption{all}
\UseComputerModernTips

\numberwithin{equation}{section}
\numberwithin{figure}{section}

\newtheorem{theorem}{Theorem}[section]
\newtheorem{lemma}[theorem]{Lemma}
\newtheorem{proposition}[theorem]{Proposition}
\newtheorem{question}[theorem]{Question}
\newtheorem{corollary}[theorem]{Corollary}

\newtheorem{example}[theorem]{Example}
\newtheorem{assumption}[theorem]{Assumption}

\makeatletter
\newtheorem*{rep@theorem}{\rep@title}
\newcommand{\newreptheorem}[2]{%
\newenvironment{rep#1}[1]{%
 \def\rep@title{#2 \ref{##1}}%
 \begin{rep@theorem}}%
 {\end{rep@theorem}}}
\makeatother
\newreptheorem{theorem}{Theorem}
\newreptheorem{proposition}{Proposition}

\theoremstyle{definition}

\newtheorem{definition}[theorem]{Definition}

\newcommand{\C}{{\mathbb{C}}}

\newcommand{\B}{\mathcal{B}}

\newcommand{\g}{\mathfrak{g}}
\newcommand{\h}{\mathfrak{h}}

\renewcommand{\b}{\mathfrak{b}}

\DeclareMathOperator{\Ad}{Ad}
\DeclareMathOperator{\ad}{ad}

\usepackage{multicol}
\usepackage{color} 
\definecolor{gold}{rgb}{0.85,.66,0}
\definecolor{cherry}{rgb}{0.9,.1,.2}
\definecolor{burgundy}{rgb}{0.8,.2,.2}
\definecolor{orangered}{rgb}{0.85,.3,0}
\definecolor{orange}{rgb}{0.85,.4,0}
\definecolor{olive}{rgb}{.45,.4,0}
\definecolor{lime}{rgb}{.6,.9,0}
\definecolor{green}{rgb}{.2,.7,0}
\definecolor{grey}{rgb}{.4,.4,.2}
\definecolor{brown}{rgb}{.4,.3,.1}

\def\PP{{\mathbb P}}
\def\grass{\mathbb{G}}
\def\fl{\mathsf{Flag}}

\def\mult{\mathsf{mult}}

\def\br{\mathsf{Br}}
\def\c{{\mathbf c}}

\def\l{\mathfrak{l}}
\def\z{\mathfrak{z}}

\newcommand{\fg}{\mathfrak{g}}
\newcommand{\fb}{\mathfrak{b}}

\def\fln{\mathsf{Flag}(n)}

\def\Bru{\mathsf{Br}}

\begin{document}

\title{Which Schubert Varieties are Hessenberg Varieties?}

\author{Laura Escobar}
\address{Department of Mathematics and Statistics\\ Washington University in St. Louis \\ One Brookings Drive \\ St. Louis, Missouri  63130 \\ U.S.A. }
\email{laurae@wustl.edu}

\author{Martha Precup}
\address{Department of Mathematics and Statistics\\ Washington University in St. Louis \\ One Brookings Drive \\ St. Louis, Missouri  63130 \\ U.S.A. }
\email{martha.precup@wustl.edu}
%\urladdr{\url{}}
%\thanks{... fill in.... }

\author{John Shareshian}
\address{Department of Mathematics and Statistics\\ Washington University in St. Louis \\ One Brookings Drive \\ St. Louis, Missouri  63130 \\ U.S.A. }
\email{jshareshian@wustl.edu}

%\keywords{} 
%\subjclass[2000]{Primary: ...; Secondary: ...}
%
%\date{\today}

%%%%%%%%%%%%%%%%%%%%
% Disclaimer
%%%%%%%%%%%%%%%%%%%%

%%%%%%%%%%%%%%%%%%%%%
%  Abstract
%%%%%%%%%%%%%%%%%%%%%

\begin{abstract}
After proving that every Schubert variety in the full flag variety of a complex reductive group $G$ is a general Hessenberg variety, we show that not all such Schubert varieties are adjoint Hessenberg varieties.  In fact, in types A and C, we provide pattern avoidance criteria implying that the proportion of Schubert varieties that are adjoint Hessenberg varieties approaches zero as the rank of $G$ increases.  We show also that in type A, some Schubert varieties are not isomorphic to any adjoint Hessenberg variety.
\end{abstract}

\maketitle

\setcounter{tocdepth}{1}
\tableofcontents

\section{Introduction}

Hessenberg varieties have been studied by applied mathematicians, combinatorialists, geometers, representation theorists, and topologists.  See \cite{Abe-Horiguchi2019} for a survey of some recent developments.  Our goal is to understand better the structure of these varieties, and in particular what restrictions on such structure exist.  To this end, we address herein a question raised by Tymoczko: 
\begin{center}Is every Schubert variety in a full flag variety a Hessenberg variety?  \end{center}
As we will discuss below, conditions known to be satisfied by the most closely studied Hessenberg varieties are also satisfied by Schubert varieties. So, Tymoczko's question is pertinent. 
There are several possible interpretations of the question.  Moreover, there are several ways to define a Hessenberg variety, giving rise to progressively more general classes of varieties.  The answer to Tymoczko's question depends on the chosen definition, as we shall see.

Let $G$ be a connected, reductive complex algebraic group.  Fix a Borel subgroup $B \leq G$ and a maximal torus $T \leq B$.  Let $\g,\b,$ and $\h$ be, respectively, the Lie algebras of $G,B,$ and $T$.  We write $Ad$ for the adjoint representation of $G$ on $\g$.

Let $N=N_G(T)$ be the normalizer of $T$ in $G$ and let $W=N/T$ be the associated Weyl group.  The flag variety $\B:=G/B$ is the union of Schubert cells $C_w$, over all $w \in W$.  For each $w$, the Schubert variety $X_w$ is the closure of $C_w$ in $\B$.

Our first definition of a Hessenberg variety is as follows.

\begin{definition} \label{ahvdef}
Given $x \in \g$ and a subspace $H$ of $\g$ such that $[\b,H] \subseteq H$, the \textbf{adjoint Hessenberg variety} $\B(x,H)$ consists of those $gB \in \B$ such that $Ad(g^{-1})(x) \in H$.
\end{definition}

Adjoint Hessenberg varieties were defined and studied by De Mari, Procesi and Shayman in \cite{DPS1992}, after being defined and studied for $G=GL_n(\C)$ only by De Mari and Shayman in \cite{DS1988}.  In both \cite{DS1988} and \cite{DPS1992}, it is assumed that $\b \subseteq H$ and that $x$ is a generic (regular semisimple) element of $\g$.  We make neither assumption here, as otherwise the topology of $\B(x,H)$ is restricted considerably.  Indeed, it follows from the results in \cite{DS1988,DPS1992} that if $\b \subseteq H$ and $x$ is regular semisimple, then the Euler characteristic $\chi(\B(x,H))$ is equal to $|W|$.  In particular, the only Schubert variety in $\B$ that is a Hessenberg variety under these assumptions is $\B$ itself.  

A larger class of varieties than that given in Definition \ref{ahvdef} is defined in \cite{GKM2006} by Goresky, Kottwitz and MacPherson, who allow an arbitrary representation of $G$, rather than restricting to the adjoint representation.

\begin{definition} \label{ghvdef}
Let $\psi:G \rightarrow GL(V)$ be a (finite-dimensional, rational) representation.  Given $x \in V$ and a $B$-invariant subspace $H$ of $V$, the \textbf{Hessenberg variety} $\B(x,H)$ consists of those $gB \in \B$ such that $\psi(g^{-1})x \in H$.
\end{definition}

We will always use the modifier ``adjoint" when referring to the Hessenberg varieties described in Definition \ref{ahvdef} and sometimes use the modifier ``general" when discussing the Hessenberg varieties described in Definition \ref{ghvdef}.  We will use repeatedly, and without reference, the fact that if  $H$ is a subspace of $\g$, then $[\b,H] \subseteq H$ if and only if $H$ is $Ad(B)$-invariant (see for example \cite[Definition 8.1.22]{Kumar}).     So, every adjoint Hessenberg variety is a general Hessenberg variety.  The subspace $H$ appearing in either definition is called a \textbf{Hessenberg space}. 

With definitions in hand, we turn to possible interpretations of Tymoczko's question, and list three precise questions.
\begin{question}[\bf the equality problem]\label{quest.E}  Is it true that for every $G,B,$ and Schubert variety $X_w \subseteq \B=G/B$, there is a Hessenberg variety $\B(x,H) \subseteq \B$ such that $X_w=\B(x,H)$?
\end{question}
\begin{question}[\bf the isomorphism problem]\label{quest.I}
Is it true that for every $G,B$ and Schubert variety $X_w \subseteq \B=G/B$, there is a Hessenberg variety $\B(x,H) \subseteq \B$ such that $X_w$ is isomorphic with $\B(x,H)$?
\end{question}
\begin{question}[\bf the general isomorphism problem]\label{quest.GI}  Is it true that for every $G,B$ and Schubert variety $X_w \subseteq G/B$, there exist a complex reductive group $G^\ast$ with Borel subgroup $B^\ast$ and a Hessenberg variety $\B^\ast(x^\ast,H^\ast) \subseteq G^\ast/B^\ast$ such that $X_w$ is isomorphic with $\B^\ast(x^\ast,H^\ast)$?
\end{question}
The answer to Question~\ref{quest.E} (and therefore to Questions~\ref{quest.I} and \ref{quest.GI}) is ``Yes" for general Hessenberg varieties, as our first main result shows.  (Relevant terminology will be discussed in Sections \ref{sec.preliminaries}  and \ref{sec.highestwt}.)

\begin{theorem}[See Theorem~\ref{dem} below]\label{dem.intro}
Let $\lambda$ be a strictly dominant weight for $G$, with associated highest weight representation $\psi:G \rightarrow GL(V(\lambda))$ and highest weight vector $v_\lambda$.  Given $w \in W$, let $\dot{w}$ be a representative of $w$ in $N$, and let $H_{w^{-1}(\lambda)}$ be the Demazure module generated by $\dot{w}^{-1}(v_\lambda)$.  Then $X_w=\B(v_\lambda,H_{w^{-1}(\lambda)})$.
\end{theorem}

Theorem~\ref{dem.intro} follows directly from a result of Bernstein, Gelfand and Gelfand in \cite{BGG}.

The situation is more interesting when we consider adjoint Hessenberg varieties.  We obtain negative results of various types.  Before describing these, we discuss why the consideration of Schubert varieties is appropriate in the study of adjoint Hessenberg varieties. 

We know of no restrictions on the structure of general Hessenberg varieties, and it is reasonable to wonder what such restrictions might exist.  In addition to Theorem \ref{dem.intro}, one can point to Section 9 of \cite{OblomkovYun}, in which various interesting curves and surfaces are shown to be general Hessenberg varieties, as evidence that such restrictions are not so easy to come by.  
A certain class of general Hessenberg varieties is examined by Chen-Vilonen-Xue~\cite{Chen-Vilonen-Xue} in their study of Springer Theory for symmetric spaces.

Adjoint Hessenberg varieties are a different matter.   \color{black} Given $x \in \g$, we consider the usual Jordan decomposition $x=x_s+x_n$ with $x_s$ semisimple and $x_n$ nilpotent.  If $x_n$ is regular in a Levi subalgebra of $\g$, then $\B(x,H)$ admits an affine paving for every Hessenberg space $H \subseteq \g$.  This is proved under the assumption $\b \subseteq H$ by Tymoczko in type A (see \cite{Tymoczko2006}) and by the second author for arbitrary $G$ (see \cite{Precup2013}), although the assumption is not necessary for the relevant arguments in either paper. When $H$ is a nilpotent subspace, Fresse proved $\B(x,H)$ is paved by affines for all $x$ when $G$ is a classical group in \cite{Fresse2016} and Xue has extended these results to groups of type $G_2$, $F_4$, and $E_6$ in \cite{Xue2020}.    So, in some sense, $\B(x,H)$ is paved by affines for ``most" $x,H$, and indeed for all $x,H$ in type A.  
This restricts considerably the structure of adjoint Hessenberg varieties.  
One can ask if there are any obvious additional restrictions.  As Schubert varieties admit affine pavings, Tymoczko's question is a good starting point in the search for such constraints.

We turn now to our results.  First, if the root system $\Phi$ for $G$ has an irreducible component not of type $A_1$ or $A_2$, then there are some $y \in W$ such that the Schubert variety $X_y \subseteq \B$ is not equal to any adjoint Hessenberg variety in $\B$.

\begin{theorem}[See Theorem~\ref{thm.not.adjoint} below]\label{thm.not.adjoint.intro}
If some nonabelian simple ideal of $\g$ is isomorphic with neither $\mathfrak{sl}_2(\C)$ nor $\mathfrak{sl}_3(\C)$, then there is some $w \in W$ such that no adjoint Hessenberg variety in $\B$ is equal to $X_w$.  In particular, assume that the root system $\Phi$ for $G$ is irreducible and not of type $A_1$ or $A_2$.  Let $\theta$ be the highest root in $\Phi$ and let $W_\theta$ be the stabilizer of $\theta$ in $W$.  Let $w_0$ be the longest element of $W$.  If $y$ is a nonidentity element of $W_\theta$, then no adjoint Hessenberg variety in $\B$ is equal to $X_{yw_0}$.
\end{theorem}

We observe in the case where $\Phi$ is irreducible, there is some function $f$ such that $[W:W_\theta] \geq f(\dim_\C\h)$ and $\lim_{n \rightarrow \infty} f(n)=\infty$.  So, Theorem \ref{thm.not.adjoint.intro} applies to an eventually negligible portion of the Schubert varieties in $\B$.  When $\Phi$ is of type A or type C, we can do much better.

\begin{theorem}[See Theorem~\ref{pattern} below]\label{pattern.intro}
Assume $G = GL_n(\C)$ or $G=SL_n(\C)$ and let $B$ be the Borel subgroup of $G$ consisting of upper triangular matrices.  Fix $w$ in the Weyl group $W=S_n$ of $G$.  If there exist $x \in \g$ and a Hessenberg space $H \subseteq \g$ such that $\B(x,H)=X_w$, then $w$ avoids the pattern $[4231]$.
\end{theorem}

It follows from Theorem \ref{pattern.intro} and the Marcus--Tardos Theorem (see \cite{MT2004}) that the number of Schubert varieties in $SL_n(\C)/B$ that are (equal to) adjoint Hessenberg varieties grows at most exponentially in $n$.  So, in type A, the portion of Schubert varieties that are equal to adjoint Hessenberg varieties is eventually negligible.  A more precise enumerative result appears in Section \ref{sec.pattern-adjoint}, along with a proof of the theorem.  Given Theorem \ref{pattern.intro}, one might hope that the set of Type A Schubert varieties that are (equal to) adjoint Hessenberg varieties is characterized by pattern avoidance.  This is not the case.  We show in Section \ref{sec.pattern-adjoint} that $X_{[14235]}$ is an adjoint Hessenberg variety in $SL_5(\C)/B_5(\C)$, but $X_{[1423]}$ is not an adjoint Hessenberg variety in $SL_4(\C)/B_4(\C)$.

We also obtain a negative answer to the isomorphism question for adjoint Hessenberg varieties in type A, although our result applies to far fewer Schubert varieties than Theorem~\ref{pattern.intro}.

\begin{theorem}[See Theorem~\ref{adjoint} below]\label{adjoint.intro} Suppose $n\geq 6$.
Assume that $G = GL_n(\C)$ or $G=SL_n(\C)$ and $B$ is the subgroup of $G$ consisting of upper triangular matrices.  Let $w_0$ be the longest element of the Weyl group $W=S_n$, and for $i \in [n-1]$, let $s_i \in W$ be the transposition $(i,i+1)$.  If $3 \leq i \leq n-3$, then there do not exist $x \in \g$ and Hessenberg space $H \subseteq \g$ such that $\B(x,H)$ is isomorphic with $X_{s_iw_0}$.
\end{theorem}

We prove in Section \ref{sec.typeA} that under the conditions given in Theorem \ref{adjoint.intro}, there is no irreducible adjoint Hessenberg variety $\B(x,H)$ with the same Betti numbers as $X_{s_iw_0}$.  In fact, we show that unless $(n,i) \in \{(8,3),(8,5)\}$, no such $\B(x,H)$ has the same Euler characteristic as $X_{s_iw_0}$.  
We use the following result, reminiscent of the restriction $\chi(\B(x,H))=|W|$ obtained by De Mari, Procesi and Shayman in \cite{DPS1992} under the assumption that $\b \subseteq H$ and $x$ is regular semisimple.

\begin{proposition}[See Proposition~\ref{eulercharprop} below]\label{eulercharprop.intro}
Let $G$ and $B$ be as in Theorem \ref{adjoint.intro}.  If the adjoint Hessenberg variety $\B(x,H)$ is irreducible and of codimension one in $\B$, then $\chi(\B(x,H))$ is divisible by $(n-2)!$.
\end{proposition}

We do not know the answer to the general isomorphism problem for adjoint Hessenberg varieties, and pose the following question.  Observe that $[653421] \in S_6$ is the ``first" Schubert variety to which Theorem \ref{adjoint.intro} applies.

\begin{question} \label{653421q}
Is there an adjoint Hessenberg variety (in a flag variety for an arbitrary reductive group) that is isomorphic to the Schubert variety $X_{[653421]}$ in the flag variety of type $A_5$?
\end{question}

One might also consider Question \ref{653421q} for the Schubert variety $X_{[4231]}$ in the flag variety of type $A_3$.  It follows from either Theorem \ref{thm.not.adjoint.intro} or Theorem \ref{pattern.intro} that $X_{[4231]}$ is not equal to any adjoint Hessenberg variety, and $X_{[4231]}$ is the ``first" type A Schubert variety to which these theorems apply.  However, there is an adjoint Hessenberg variety in $SL_4(\C)/B$ with the same Betti numbers as $X_{[4231]}$ (see Example~\ref{ex.Betti} below).  

We have a type C version of Theorem \ref{pattern.intro}.  Consider the embedding $\phi:Sp_{2n}(\C) \rightarrow SL_{2n}(\C)$ whose image stabilizes the alternating form $\langle .,. \rangle$ given by  
$$
\langle e_i,e_j \rangle=\left\{ \begin{array}{cc} 1 & i<j=2n+1-i, \\ -1 & i>j=2n+1-i \\ 0 & \mbox{otherwise}. \end{array} \right.
$$ 
Here $e_1,\ldots,e_{2n}$ denotes the standard basis of $\C^{2n}$. One obtains from this the embedding $\phi^\ast$ of the type C Weyl group into $S_{2n}$ whose image consists of those permutations $w$ satisfying $w_{2n+1-i}=2n+1-w_i$ for all $i$.

\begin{theorem}[See Theorem~\ref{patternc} below]\label{patternc.intro}
Let $G=Sp_{2n}(\C)$ and let $B \leq G$ be the Borel subgroup whose image under $\phi$ consists of upper triangular matrices.  Fix $w \in W$.  If there exist $x \in \g=\mathfrak{sp}_{2n}(\C)$ and a subspace $H$ of $\g$ such that $[\b,H] \subseteq H$ and $\B(x,H)=X_w$, then $\phi^\ast(w)$ avoids the pattern $[4231]$.
\end{theorem}

Our proof of Theorem \ref{patternc.intro} utilizes an interesting fact relating adjoint Hessenberg varieties in types A and C.  Let $E$ be the linear transformation on $\C^{2n}$ such that $\langle v,w \rangle=v^{\mathsf{\mathsf{tr}}}Ew$ for all $v,w \in \C^{2n}$, and define the automorphism $\sigma$ of $G_A:=SL_{2n}(\C)$ by
$$
\sigma(A)=E(A^{\mathsf{\mathsf{tr}}})^{-1}E^{-1}.
$$
Then $G_C:=\phi(Sp_{2n}(\C))$ is the group of $\sigma$-fixed points in $G_A$.  Moreover, if $B_A$ is the group of upper triangular matrices in $G_A$, then $\sigma$ fixes $B_A$ setwise, and $B_C:=(B_A)^\sigma$ is the image under $\phi$ of the Borel subgroup of $Sp_{2n}(\C)$ described in Theorem \ref{patternc.intro}.  We observe that the action of $\sigma$ on $G_A$ induces an automorphism of the variety $\B_A:=G_A/B_A$.  The map $\phi^\prime$ from $\B_C:=G_C/B_C$ to $\B_A:=G_A/B_A$ sending $gB_C$ to $\phi(g)B_A$ is a well-defined embedding, and $(\B_A)^\sigma=\phi^\prime(\B_C)$ (see for example \cite[Proposition 6.1.1.1]{Lakshmibai-Raghavan}).

\begin{theorem}[See Theorem \ref{thm: C to A} below] \label{foldhess}
If ${\mathcal V}_C$ is an adjoint Hessenberg variety in the type C flag variety $\B_C$, then there is some adjoint Hessenberg variety ${\mathcal V}_A$ in the type A flag variety $\B_A$ such that $\phi^\prime({\mathcal V}_C)=({\mathcal V}_A)^\sigma$.
\end{theorem}

The paper follows the outline presented above.  Section~\ref{sec.preliminaries} is devoted to background, notation, and terminology.  Our study of general Hessenberg varieties and a proof of Theorem~\ref{dem.intro} can be found in Section~\ref{sec.highestwt}, in which we prove also that the answer to the equality question (Question~\ref{quest.E}) is ``No'' for adjoint Hessenberg varieties in all Lie types.  We then focus our attention on adjoint Hessenberg varieties in the type A flag variety in Sections~\ref{sec.pattern-adjoint} and \ref{sec.typeA}, proving Theorems~\ref{pattern.intro} and~\ref{adjoint.intro}.  Finally, our study of adjoint Hessenberg varieties in the type C flag variety is undertaken in Section~\ref{sec.typeC}, in which we prove Theorems~\ref{patternc.intro} and~\ref{foldhess} .

\

\textbf{Acknowledgements:} We learned of Tymoczko's question regarding Schubert and Hessenberg varieties during an open problem session at BIRS Workshop 18w5130 on \textit{Hessenberg Varieties in Combinatorics, Geometry and Representation Theory} in October 2018.  The second and third authors are grateful for the hospitality of the Banff International Research Station and organizers of that workshop. We thank Reuven Hodges, Allen Knutson, and Jenna Rajchgot for helpful discussions. Escobar is partially supported by NSF Grant DMS 1855598. Precup is partially supported by NSF Grant DMS 1954001.  Shareshian is partially supported by NSF Grant DMS 1518389.

\

%%%%%%%%%%%%%

\section{Definitions, Notation, and Preliminary Results} \label{sec.preliminaries}

We review here various known results about algebraic groups, Weyl groups, flag varieties, and Hessenberg varieties.  A reader familiar with basic facts about these objects can skip this section and refer back when necessary.  Facts stated without reference or argument can be found in at least one of \cite{Bjorner-Brenti, Borel-LAG, Fulton-Harris, Humphreys-LieAlg}.  

As in the introduction, $G$ is a reductive algebraic group with Borel subgroup $B$ and maximal torus $T \leq B$.  The Lie algebras of $G,B,T$ are denoted, respectively, by $\fg, \fb,\h$.  The Weyl group of $G$ is $W:=N_G(T)/T$.  We fix a representative $\dot{w}\in N_G(T)$ for each Weyl group element $w \in W$.

\subsection{The root system and the Bruhat order} The Lie algebra $\fg$ admits a Cartan decomposition,
$$
\g=\h \oplus \bigoplus_{\gamma \in \Phi}\g_\gamma.
$$ 
Here $\Phi \subseteq \h^\ast$ is the root system for $\fg$, and each root space $\fg_\gamma$ is a $1$-dimensional subspace of $\fg$ satisfying
$
[h,x]=\gamma(h)x
$
whenever $h \in \h$ and $x \in \g_\gamma$.  We can choose a set of simple roots $\Delta \subseteq \Phi$ such that each $\gamma \in \Phi$ is either a non-negative linear combination of elements of $\Delta$ or a non-positive such combination.  This gives a decomposition $\Phi=\Phi^+ \sqcup \Phi^-$.  We may (and do) choose $\Delta$ so that
$$
\b=\h \oplus \bigoplus_{\gamma \in \Phi^+}\g_\gamma.
$$
There is a partial order on $\Phi$ given by
	\begin{equation}\label{eq-root-order}
	\beta \preceq \gamma \textup{ whenever } \gamma-\beta \textup{ is a nonnegative linear combination of positive roots.}
	\end{equation}
Each finite, irreducible root system $\Phi$ contains a unique maximal element with respect to this order called the \textbf{highest root} of $\Phi$ and denoted herein by $\theta\in \Phi^+$.  

If $\fg=\mathfrak{sl}_n(\C)$ and $\h$ the Cartan subalgebra of diagonal matrices, we write $\Phi=\{ \epsilon_i-\epsilon_j : 1\leq i,j \leq n \}$ with positive roots $\Phi^+ = \{ \epsilon_i -\epsilon_j\in \Phi : i<j \}$.  Here $\epsilon_i: \h \to \C$ denotes projection to the $i$-th diagonal entry. We assume furthermore that $\Delta = \{ \epsilon_i - \epsilon_{i+1} : 1\leq i \leq n-1 \}$.  The highest root of $\Phi$ in this case is $\theta = \epsilon_1-\epsilon_n$.

The restriction of the adjoint action of $G$ on $\g$ to $N_G(T)$ preserves the Cartan decomposition and factors through $T$.  Thus we get an action of $W$ on $\h^\ast$ that restricts to an action on $\Phi$.  For each $\alpha \in \Delta$, there is the simple reflection $s_\alpha \in W$, which acts on $\h^\ast$ as the reflection through the hyperplane orthogonal to $\alpha$.  In the type A case, if $\alpha= \epsilon_i -\epsilon_{i+1}$ we write $s_\alpha=s_i$ for the corresponding simple transposition exchanging $i$ and $i+1$ in $W= S_n$.   The set $S$ of simple reflections generates $W$, and the length $\ell(w)$ of $w \in W$ is the shortest length of a list of elements of $S$ (called a reduced word) whose product is~$w$.

The \textbf{Bruhat order} $\leq_\br$ is the partial order on $W$ defined by $v \leq_\br w$ if some reduced word for $v$ is a subword of some reduced word for~$w$. We write $w_0$ for the unique maximal element of $W$ in the Bruhat order.

\subsection{The flag variety and Schubert varieties} The group $G$ acts on the flag variety $\B:=G/B$ by translation.  Each $B$-orbit in this action contains exactly one coset $\dot{w}B$ ($w \in W$), and so
$$
\B=\bigsqcup_{w \in W} B\dot{w}B/B.
$$
We write $C_w$ for $B\dot{w}B/B$, called a \textbf{Schubert cell}.  Each Schubert cell is isomorphic to the affine space $\C^{\ell(w)}$.  The closure of $C_w$ in $G/B$ is the \textbf{Schubert variety} $X_w:= \overline{C_w}$.  We have that
$$
X_w=\bigsqcup_{v \leq_\br w}C_v.
$$

\subsection{Adjoint Hessenberg varieties} Let $H\subseteq \fg$ be a {Hessenberg space}, that is, a subspace such that $[\fb, H] \subseteq H$, and fix $x\in \fg$.  Equivalently, $H\subseteq \fg$ is $B$-invariant with respect to the Adjoint action (see for example \cite[Definition 8.1.22]{Kumar}).   Given a root $\gamma\in \Phi$, let $\pi_\gamma$ denote the projection of $\fg$ to the root space $\fg_\gamma$.  As $\dim_\C \fg_\gamma =1$ and the $T$-module $\h \oplus \bigoplus_{\beta \in \Phi\setminus \{\gamma\}} \fg_\gamma$ has no quotient isomorphic to $\fg_\gamma$, it follows from basic facts about direct sums that, for any Hessenberg space $H$ in $\fg$, 
\begin{eqnarray}\label{eqn.Hproj}
\textup{ if } \ \pi_\gamma (H)\neq 0 \  \textup{ then } \ \fg_\gamma \subseteq H.
\end{eqnarray}
We will use~\eqref{eqn.Hproj} and its consequences repeatedly below, frequently without reference.

As in the introduction, we define the adjoint Hessenberg variety corresponding to $x$ and Hessenberg space $H$ by
\[
\B(x,H) := \{gB\in \B : g^{-1} \cdot x \in H\}.
\]
Here $g\cdot x:= Ad(g)(x)$.

\subsection{Type A: The tableau criterion, pattern avoidance, and flags in $\C^n$} \label{sec.typeA-setup} Now we consider the case $G=SL_n(\C)$ or $G=GL_n(\C)$.  We record here various results that will be used below.
In this case, $W$ is isomorphic to the symmetric group $S_n$.  We write elements of $S_n$ in one-line notation, $$w=[w_1 w_1 \cdots  w_n]$$ where $w_i=w(i)$ in the natural action of $w$ on $[n]:=\{1,2,\ldots,n\}$.  

Given $w \in S_n$ and $1 \leq j \leq k \leq n$, we write $I_{j,k}(w)$ for the $j^{th}$ smallest element of $\{w_i : 1 \leq i \leq k\}$.  So, for example, if $w=[52341] \in S_5$, then $I_{2,4}(w)=3$, as $3$ is the second smallest element of $\{5,2,3,4\}$.  The following characterization of the Bruhat order can be found in~\cite{Bjorner-Brenti}.

\begin{theorem}[Tableau Criterion] \label{tc}
Let $v,w\in S_n$.  Then $v \leq_\br w$ if and only if $I_{j,k}(v) \leq I_{j,k}(w)$ for all $1 \leq j \leq k \leq n$.
\end{theorem}

Given $v \in S_m$ and $w \in S_n$ with $m\leq n$, we say that $w$ \textbf{contains the pattern} $v$ if there exist $1 \leq i_1<i_2<\ldots<i_m \leq n$ such that, for all $j,k \in [m]$, $w_{i_j}<w_{i_k}$ if and only if $v_j<v_k$.  So, for example, $[631524]\in S_6$ contains the pattern $[4231]\in S_4$, realized by the subsequence $6352$.  We say $w$ \textbf{avoids} $v$ if $w$ does not contain the pattern $v$.   For example, $[631524]$ avoids $[4321]$.  We write $S_n(v)$ for the set of all $w \in S_n$ avoiding $v$.  Marcus and Tardos proved in \cite{MT2004} that (as conjectured independently by Stanley and Wilf) for every fixed $v$, $|S_n(v)|$ grows exponentially with $n$.  It follows immediately that
\begin{equation} \label{sparse}
\lim_{n \rightarrow \infty} \frac{|S_n(v)|}{|S_n|}=0.
\end{equation}

In this type A setting, we take $B$ to be the Borel subgroup consisting of upper triangular matrices in $G$.  Write $e_1,e_2,\ldots,e_n$ for the standard basis of $\C^n$.  We set $F_k:=\C\{e_j : j \in [k]\}$, and observe that $B$ is the stabilizer in $G$ of the flag
$$
{\mathcal F}_\bullet:=( 0=F_0\subset F_1\subset \ldots \subset F_n=\C^n).
$$
As $G$ acts transitively on the set 
\[
\fln:= \{ \mathcal{V}_\bullet = (0=V_0 \subset V_1\subset \ldots \subset V_n=\C^n) : \dim_\C V_i = i \}
\] 
of all full flags in $\C^n$, we obtain a bijection $\B=G/B \to \fln$ defined by
\begin{eqnarray}\label{eqn.cosets-to-flags}
gB \mapsto g{\mathcal F}_\bullet:=(0=gF_0\subset gF_1\subset \ldots\subset gF_n=\C^n).
\end{eqnarray}

Let us assume (temporarily) that a Hessenberg space $H$ contains the Borel algebra $\b$ of upper triangular matrices in $\g$.  In this context, we define the \textbf{Hessenberg vector} $h=h(H):=(h_1,\ldots,h_n)$ by setting $h_j$ to be the largest integer $i>j$ such that the span $\C \{E_{ij}\}$ of the elementary matrix $E_{ij}$ is contained in $H$, if such $i$ exists, and $h_j=j$ otherwise.  We observe that the sequence $h$ is weakly increasing (as $H$ is $Ad(B)$-invariant) and satisfies $h_j \geq j$ for all $j$ (as $\fb\subseteq H$).  Moreover, the Hessenberg space $H$ is determined by, and uniquely determines, the Hessenberg vector $h(H)$. A direct computation shows that under the restriction of the bijection from~\eqref{eqn.cosets-to-flags} to $\B(x, H)$ we have 
\begin{equation} \label{flagmodel}
\B(x ,H ) \simeq \{{\mathcal V}_\bullet \in \fln : xV_i \subseteq V_{h_i} \mbox{ for all } i \in [n] \mbox{ where } h =h(H)\}.
\end{equation}
We will make use this identification below whenever it is convenient to do so.  

We will also consider the image of a Schubert variety $X_w=\overline{B\dot{w}B/B}$ under this identification.  To this end, 
for $w \in S_n$ and $p,q \in [n]$, we set
$$
r_{p,q}(w):=\left|\{i \in [p]:w_i \in [q]\}\right|.
$$
Then the correspondence between cosets and flags maps $X_w$ to the set of all ${\mathcal V}_\bullet$ satisfying
$$
\dim_\C (V_p \cap F_q) \geq r_{p,q}(w)
$$
for all $p,q \in [n]$, see \cite[\S 10.5]{Fulton}.

\subsection{Representations of reductive groups}\label{sec.reps} Returning to the setting of an arbitrary reductive algebraic group $G$, let $\psi:G \rightarrow GL(V)$ be a (rational, finite-dimensional) representation. Then the differential $d\psi:\g \rightarrow \mathfrak{gl}(V)$ is a Lie algebra homomorphism.  Both $\psi(T)$ and $d\psi(\h)$ are diagonalizable and thus there exist (finitely many) \textbf{weights} $\lambda \in \h^\ast$ such that
\begin{itemize}
\item $V=\bigoplus_\lambda V_\lambda$, where
\item $d\psi(h)(v)=\lambda(h) v$ for all $v \in V_\lambda$ and $h \in \h$, and
\item $\psi(t)$ acts as a scalar transformation on $V_\lambda$ for all $t \in T$ and all $\lambda$.
\end{itemize}
The subspaces $V_\lambda$ are called \textbf{weight spaces} of $V$.  We write $\Lambda$ for the set of all weights of all representations of $G$, which forms a lattice in $\h^\ast$.  The set of weights of the adjoint representation $Ad:G \rightarrow GL(\g)$ is $\Phi \sqcup \{0\}$, hence $\Phi \subseteq \Lambda$.  

If $\psi$ is irreducible, then there is a unique $1$-dimensional $\psi(B)$-invariant subspace of $V$, which is a weight space.  The associated weight $\lambda$ is called the \textbf{highest weight} of $\psi$, and we write $V(\lambda)$ for $V$.  Any nonzero $v_\lambda \in V(\lambda)_\lambda$ is a \textbf{highest weight vector} in $V(\lambda)$.

A weight $\lambda \in \Lambda$ is \textbf{dominant} if $\lambda$ is the highest weight for some irreducible representation.  The action of $W$ on $\h^\ast$ induces an action of $W$ on $\Lambda$.  We observe that $\dim_\C V_\lambda=\dim_\C V_{w(\lambda)}$ for every weight $\lambda$ for $V$ and every $w \in W$.  A dominant weight $\lambda$ is \textbf{strictly dominant} if the stabilizer of $\lambda$ in $W$ is trivial.  The partial order $\preceq$ on $\Phi$ defined above in~\eqref{eq-root-order} extends to a partial order on $\Lambda$.  We write $\rho \preceq \lambda$ if $\lambda-\rho$ is a nonnegative linear combination of positive roots.

We remark that every representation $\psi'$ of $\g$ is of the form $\psi'=d\psi$ for some representation $\psi$ of $G$, and that such a $\psi'$ is irreducible if and only if $\psi$ is irreducible.  Thus we may also use the terminology defined above when referring to representations of the Lie algebra $\g$.  Finally, a representation $\psi':\g\rightarrow \mathfrak{gl}(V)$ makes $V$ a module for the universal enveloping algebra $U(\g)$, and we make no distinction between such representations and modules.
Given dominant $\lambda\in\Lambda$, we refer to $V(\lambda)$ as a \textbf{highest weight module} for $G$, or equivalently $U(\fg)$.

%%%%%%%%%%%%%

\section{Highest weight Hessenberg varieties} \label{sec.highestwt}

In this section we study Questions~\ref{quest.E}, \ref{quest.I}, and \ref{quest.GI} for \textit{general} Hessenberg varieties of Definition~\ref{ghvdef}.  We prove Theorem~\ref{dem.intro} (see Theorem~\ref{dem} below), establishing that the answer to all three questions in this context is ``Yes.''  
Then we show that the answer to the equality problem, Question~\ref{quest.E}, for adjoint Hessenberg varieties is ``No'', by proving Theorem \ref{thm.not.adjoint.intro} (see Theorem~\ref{thm.not.adjoint} below).

Let $\psi: G \to GL(V(\lambda))$ be the irreducible representation of $G$ with highest weight $\lambda$, and fix a highest weight vector $v_\lambda\in V(\lambda)_\lambda$. Since $\C\{v_\lambda\}$ is a $\psi(B)$-invariant subspace of $V(\lambda)$, it follows that the Hessenberg variety $\B(v_\lambda, H)$ is invariant under left translation by $B$ and therefore a union of Schubert varieties.  We call the Hessenberg variety $\B(v_\lambda, H)$ a \textbf{highest weight Hessenberg variety}.  

Highest weight Hessenberg varieties defined using the adjoint representation have been studied by Tymoczko and by Abe--Crooks in\cite{Tymoczko2006A, Abe-Crooks2016}.  In the adjoint case, the highest weight is the highest root $\theta\in \Phi$.  
We fix nonzero $E_\theta \in \g_\theta$, so $E_\theta$ is a highest weight vector for the adjoint representation of $G$.
Abe and Crooks give an explicit description of of the highest weight Hessenberg variety $\B(E_\theta, H)$ as a union of Schubert varieties in the type A case whenever $\fb\subseteq H$.  The following result, due to Tymoczko (see~\cite[Prop.~4.5]{Tymoczko2006}), describes a collection of Schubert varieties equal to highest weight adjoint Hessenberg varieties. 

\begin{proposition}[Tymoczko]\label{prop.highest-wt-adjoint} Suppose $\gamma\in \Phi$ is a root in the same $W$-orbit as $\theta$ and let $w$ be the maximal length element of $W$ such that $w^{-1}(\theta)=\gamma$.  Let $H_\gamma$ be the $B$-submodule of $\fg$ generated by $E_{w^{-1}(\theta)} = E_\gamma$.  Then $X_w = \B(E_\theta, H_\gamma)$.
\end{proposition}

The $B$-module $H_\gamma$ defined in Proposition \ref{prop.highest-wt-adjoint} is known as a Demazure module.  Such modules are defined similarly for arbitrary irreducible representations of $G$, and will be used below to prove Theorem~\ref{dem.intro}.

\subsection{Demazure Modules} \label{sec.demazure}

Throughout this section, we let $\lambda$ denote a fixed dominant weight and $V(\lambda)$ the associated highest weight module for $G$, or equivalently, for $U(\fg)$. For each $w \in W$, fix a nonzero vector $v_{w(\lambda)}$ in the (one-dimensional) weight space $V(\lambda)_{w(\lambda)}$.  The \textbf{Demazure module} $H_{w(\lambda)}$ is the $U(\b)$-submodule of $V(\lambda)$ generated by $v_{w(\lambda)}$.  As remarked in \cite[Definition 8.1.22]{Kumar}, $H_{w(\lambda)}$ is $B$-invariant and so is a Hessenberg space in $V(\lambda)$.
 Indeed, $H_{w(\lambda)}$ is the $B$-submodule of $V(\lambda)$ generated by $v_{w(\lambda)}$.

\begin{example}\label{ex-adjoint-demazure}
We record here some observations regarding these constructions in the adjoint case that will be useful in Sections~\ref{sec.pattern-adjoint} and \ref{sec.typeC}.
Given a root $\gamma\in \Phi$ in the same $W$-orbit as $\theta$, the Demazure module $H_\gamma\subseteq \fg$ is the $B$-submodule generated by a root vector $E_\gamma\in \fg_\gamma$.
Since each $H_\gamma$ is $B$-invariant, each is also $T$-invariant and it follows that $\fg_\beta \cap H_{\gamma}\neq \varnothing$ implies $\fg_\beta\subset H_\gamma$.

Motivated by this property, we define a second partial order on $\Phi$ by 
\begin{eqnarray}\label{eqn.Demazure-order}
\gamma \leq \beta \ \textup{ whenever } \  \fg_\beta \subseteq H_\gamma. 
\end{eqnarray}
We note that $\gamma \leq \beta$ implies $\gamma \preceq \beta$, where $\preceq$ is defined in \eqref{eq-root-order}, but not vice versa.
For example, say $\g=\mathfrak{sl}_4(\C)$ so $\Phi$ is of type $A_3$.  Set
\[
H_{21}:= \mathfrak{sl}_4(\C)\cap \C\{ E_{ij} : i\in \{1,2\}   \},
\]
Then $H_{\epsilon_2-\epsilon_1} \subseteq H_{21}$, since conjugation by any upper triangular matrix $b\in B$ maps the root vector $E_{21}\in \fg_{\epsilon_2-\epsilon_1}$ into $H_{21}$.  Since $E_{34}\notin H_{21}$, this shows $\epsilon_2-\epsilon_1 \nleq \epsilon_3-\epsilon_4$.  On the other hand, $\epsilon_2-\epsilon_1 \preceq \epsilon_3-\epsilon_4$,
 since $\epsilon_2-\epsilon_1\in \Phi^-$ and $\epsilon_3-\epsilon_4\in\Phi^+$.
\end{example}

The partial order $\leq$ on $\Phi$ defined in~\eqref{eqn.Demazure-order} is a special case of the partial order on $\{w(\lambda): w\in W \}$ for $\lambda\in \Lambda$ a dominant weight defined and studied by Proctor in~\cite{Proctor1982}.  The following lemma summarizes \cite[Proposition 3]{Proctor1982} in the adjoint case.

\begin{lemma}[Proctor]\label{lem.partial.order} Given $\gamma, \gamma'\in \Phi$ we have $\gamma \leq \gamma'$ if and only if there exist positive roots $\gamma_{i_1}, \ldots, \gamma_{i_k}\in \Phi^+$ and positive integers $n_1, n_2, \ldots, n_k$ such that $\gamma'= \gamma+n_1\gamma_{i_1}+n_2\gamma_{i_2}+\cdots + n_k \gamma_{i_k}$ and $\gamma+n_1\gamma{i_1}+\cdots + n_m \gamma_{i_m}\in \Phi$ for all $1\leq m \leq k$.  
\end{lemma}

\subsection{Proof of Theorem~\ref{dem.intro}} \label{sec-proof-3.4}
Let $J$ denote the set of simple reflections stabilizing $\lambda$, i.e., $J=\{s_\alpha: s_\alpha(\lambda)=\lambda\}$.  The subgroup $W_J$ generated by $J$ is the stabilizer of $\lambda$ in $W$.  Recall that the set of left cosets of $W_J$ in $W$ can be identified with the set of shortest left coset representatives, denoted herein by $W^J$ (see \cite[Section 2.4]{Bjorner-Brenti}).  The Bruhat order on $W$ induces an order on $W^J$, and we have $\tau \leq_\br w$ implies $\tau^J\leq_\br w^J$ where $\tau^J$ and $w^J$ denote the shortest coset representatives for $\tau W_J$ and $w W_J$, respectively (see \cite[Proposition 2.5.1]{Bjorner-Brenti}).    Note that there is a bijection between $W^J$ and the set $\{w(\lambda) : w\in W\}$ given by $\tau \mapsto \tau(\lambda)$ and thus $H_{w(\lambda)} = H_{w^J(\lambda)}$.

The next result is essentially Theorem 2.9 of the paper \cite{BGG} of Bernstein, Gel'fand and Gel'fand.  In \cite{BGG} it is assumed that $\g$ is the Lie algebra of a simply connected semisimple group, that $H_{w(\lambda)}$ is the module for the nilpotent radical of $\b$ generated by $v_{w(\lambda)}$, and that $\lambda$ is strictly dominant.  The proof of \cite[Theorem 2.9]{BGG} given therein remains valid under the weaker assumptions stated in Theorem \ref{bggthm} below.

\begin{theorem}[Bern\v{s}te\u{\i}n-Gel{'}fand-Gel{'}fand] \label{bggthm}
Suppose $\lambda$ is a dominant weight for the reductive Lie algebra $\g$, with associated irreducible representation $\psi':\g \rightarrow \mathfrak{gl}(V(\lambda))$ and highest weight vector $v_\lambda$.  Let $\tau,w$ be distinct elements of $W^J$.  Then $\tau  <_\br w$ if and only if $H_{\tau(\lambda)} \subset H_{w(\lambda)}$.
\end{theorem}

With this terminology in place, we can now prove the following result, generalizing Proposition~\ref{prop.highest-wt-adjoint} above.

\begin{theorem}\label{thm.general-highest-wt} Suppose $\lambda$ is a dominant weight for $G$, with associated highest weight representation $\psi: G \to GL(V(\lambda))$.  Let $\mu$ be a weight of $V(\lambda)$ such that $\mu$ and $\lambda$ are in the same $W$-orbit and let $w^{}\in W$ be the longest element satisfying $w (\lambda)=\mu$.  Then $X_{w^{-1}} = \B(v_\lambda, H_{w^J(\lambda)})$.
\end{theorem}

\begin{proof} Our assumptions on $w$ imply $w = w^Jy_0$ where $w^J\in W^J$ and $y_0$ is the longest element of $W_J$.  We have $\tau \leq_\br w$ if and only if $\tau^J \leq_\br w^J$ in this case.

As $\C \{v_\lambda\}$ is $\psi(B)$-invariant, so is the general highest weight Hessenberg variety $\B(v_\lambda,H_{w^{J}(\lambda)})$.  
Therefore $\B(v_\lambda,H_{w^{J}(\lambda)})$ is a union of Schubert cells.  We see now that
\begin{equation} \label{whichsigma}
\B(v_\lambda,H_{w^{J}(\lambda)})=\bigcup_{\substack{\tau\in W\\ \psi(\dot{\tau})(v_\lambda) \in H_{w^J(\lambda)}} } C_{\tau^{-1}}.
\end{equation}
Note that $\psi(\dot{\tau})(v_\lambda) \in H_{w^{J}(\lambda)} $ if and only if $H_{\tau^{J}(\lambda)} \subseteq H_{w^{J}(\lambda)}$.  
Combining \eqref{whichsigma} and the fact that $\tau \leq_\br w$ if and only if $\tau^{-1} \leq_\br w^{-1}$ with Theorem \ref{bggthm} (applied to $\psi'=d\psi$), we see that 
$$
\B(v_\lambda,H_{w^{J}(\lambda)})=X_{w^{-1}},
$$
as desired.
\end{proof}

We are ready to prove Theorem \ref{dem.intro}, restated here for convenience.

\begin{theorem}\label{dem}
Let $\lambda$ be a strictly dominant weight for $G$, with associated highest weight representation $\psi:G \rightarrow GL(V(\lambda))$ and highest weight vector $v_\lambda$.   Then $X_w=\B(v_\lambda,H_{w^{-1}(\lambda)})$.
\end{theorem}

\begin{proof}Suppose $\lambda$ is a strictly dominant weight of $G$.  Then $J=\{e\}$ and $W^J=W$.  Theorem~\ref{dem} now follows immediately from Theorem~\ref{thm.general-highest-wt} since the weights $w(\lambda)$ for $w\in W$ are pairwise distinct.
\end{proof}

%%%%%%%%%%%%%%%%%%%%

\subsection{Schubert varieties that are not adjoint Hessenberg varieties} 

We prove herein that for ``most" reductive $G$ there are Schubert varieties in $\B$ that are not Hessenberg varieties.  The following is a restatement of Theorem~\ref{thm.not.adjoint.intro} above.

\begin{theorem} \label{thm.not.adjoint}
If some nonabelian simple ideal of $\g$ is isomorphic with neither $\mathfrak{sl}_2(\C)$ nor $\mathfrak{sl}_3(\C)$, then there is some $w \in W$ such that no adjoint Hessenberg variety in $\B$ is equal to $X_w$.  In particular, assume that the root system $\Phi$ for $G$ is irreducible and not of type $A_1$ or $A_2$.  Let $\theta$ be the highest root in $\Phi$ and let $W_\theta$ be the stabilizer of $\theta$ in $W$.  Let $w_0$ be the longest element of $W$.  If $y$ is a nonidentity element of $W_\theta$, then no adjoint Hessenberg variety in $\B$ is equal to $X_{yw_0}$.
\end{theorem}

We begin by recalling some well-known facts about the  structure of reductive groups.  Given reductive $G$, let $Z^0$ be the connected component of the identity in the center $Z(G)$.  Then $G=Z^0G^\prime$, with the commutator subgroup $G^\prime \leq G$ being semsimple.  Moreover, $Z^0 \cap G^\prime$ is finite. (See for example \cite[Proposition 7.3.1, Corollary 8.1.6]{Springer}.)  As $G^\prime$ is semisimple, there exist nonabelian simple algebraic groups $L_1,\ldots,L_k$ such that $G^\prime$ is the central product $L_1 \circ \ldots \circ L_k$ (see for example \cite[Theorem 8.1.5]{Springer}).  Set $K:=Z^0 \cap G^\prime$ and $\overline{G}:=G/K \cong G^\prime/K \times Z^0/K$.  The projection of $G$ onto $\overline{G}$ induces a surjection of Lie algebras, which has trivial kernel as the finite group $K$ has trivial Lie algebra.  So, if $\g$, $\g^\prime$ and $\z$ are the respective Lie algebras of $G$, $G^\prime$ and $Z^0$, then $\g = \z \oplus \g^\prime$.  Each $L_i$ has finite center $Z(L_i) \leq Z(G^\prime)$, and $Z(G^\prime)=\prod_{i=1}^k Z(L_i)$.  For $i \in [k]$, let $\l_i$ be the Lie algebra of $L_i$.  Considering the projection of $G^\prime$ onto $G^\prime/Z(G^\prime) \cong L_1/Z(L_i) \times \ldots \times L_k/Z(L_k)$, we see that $\g^\prime=\bigoplus_{i=1}^k \l_i$.  The upshot of all this is that
\begin{itemize}
\item $\g=\z \oplus \bigoplus_{i=1}^k \l_i$;
\item if $i,j \in [k]$ with $i \neq j$, then $Ad(L_i)$ acts trivially on both $\z$ and $\l_j$;
\item if, for each $i \in [k]$, $\l_i=\h_i \oplus \bigoplus_{\alpha \in \Phi_i} \l_{i,\alpha}$ is a Cartan decomposition of $\l_i$, then $\h:=\z \oplus \bigoplus_{i=1}^k \h_i$ is a Cartan subalgebra of $\g$, and $\g=\h \oplus \bigoplus_{i=1}^k \bigoplus_{\alpha \in \Phi_i} \l_{i,\alpha}$ is a Cartan decomposition.
\end{itemize}
For each $i$, let $T_i$ and $B_i$ be, respectively, a maximal torus and Borel subgroup of $L_i$ with $T_i \leq T \cap B_i$ and $B_i \leq B$.  We may assume that the Cartan subalgebra $\h_i$ described above is the Lie algebra of $T_i$, that $\b_i:=\h_i \oplus \bigoplus_{\alpha \in \Phi_i^+}\l_{i,\alpha}$ is the Lie algebra of $B_i$, and that $\Phi=\bigcup_{i=1}^k \Phi_i$.  Set $W_i=N_{L_i}(T_i)/T_i$. Then $W=\prod_{i=1}^k W_i$.  Let $\theta_i$ be the highest root in the (irreducible, crystallographic) root system $\Phi_i$.

\begin{lemma} \label{lemma.highestrootstabilizer}
Let $\Gamma$ be an irreducible, crystallographic root system with associated Weyl group $X$, and let $\alpha \in \Gamma$.  The stabilizer of $\alpha$ in $X$ is trivial if and only if $\Gamma$ is of type $A_1$ or $A_2$.
\end{lemma}

\begin{proof}
 We observe that $\alpha \in \Gamma$ has trivial stabilizer in $X$ if and only if the $X$-orbit $X(\alpha)$ has size $|X|$.  If $\Gamma$ has rank three or more, then $|X(\alpha)| \leq |\Gamma|<|W|$ (see for example \cite[p. 43]{Carter}).  If $\Gamma$ is of type $B_2$ or $G_2$, then $|\Gamma|=|X|$, but since $\Gamma$ contains roots of two different lengths, $|X(\alpha)|<|\Gamma|$.  Inspection shows that if $\Gamma$ is of type $A_1$ or $A_2$ then every root in $\Gamma$ has trivial stabilizer in $X$. 
\end{proof}

\begin{lemma} \label{liekolchinlemma}
Let $\lambda$ be a strictly dominant weight for $G$, with associated highest weight representation $\psi:G \rightarrow GL(V(\lambda))$ and $0 \neq v \in V(\lambda)$. Then there is some $b \in B$ such that $\psi(b)v$ projects nontrivially onto $V(\lambda)_\lambda$.
\end{lemma}

\begin{proof}
By the Lie-Kolchin Theorem (see for example \cite[Theorem 17.6]{Humphreys-LAG}) there is some $1$-dimensional $B$-invariant subspace of the $B$-submodule of $V(\lambda)$ generated by $v$.  The only $1$-dimensional $B$-invariant subspace of $V(\lambda)$ is the highest weight space $V(\lambda)_\lambda$.
\end{proof}

We are ready to prove Theorem \ref{thm.not.adjoint}.  

\begin{proof}[Proof of Theorem~\ref{thm.not.adjoint}.] We pick $j \in [k]$ such that $\Phi_j$ is neither of type $A_1$ nor of type $A_2$.  Applying Lemma \ref{lemma.highestrootstabilizer} with $(\Gamma,X,\alpha)=(\Phi_j,W_j,\theta_j)$, we see that $\theta_j$ has nontrivial stabilizer in $W_j$.  So, fix $1 \neq y \in W_j$ with $y(\theta_j)=\theta_j$.  Let $w_{0,j}$ be the longest element of $W_j$.  Note that $\ell(yw_{0,j})<\ell(w_{0,j})$ so in particular, $w_{0,j} \not\leq_{\br} yw_{0,j}$.  

Assume for contradiction that there exist $x \in \g$ and a Hessenberg space $H \subseteq \g$ such that $X_{yw_{0,j}}=\B(x,H)$. Write $$x=x_0+\sum_{i=1}^k x_i,$$ with $x_0 \in \z$ and $x_i \in \l_i$ for all $i \in [k]$.  As $B=eB \in X_{yw_{0,j}}$, we see that $x \in H$.  If $x_j=0$ then $\dot{w} \cdot x=x$ for all $w \in W_j$. In particular, $\dot{w}_{0,j}B \in \B(x,H)$, which is impossible since $w_{0,j} \not\leq_\br yw_{0,j}$.  So, we assume now that $x_j \neq 0$. 

By Lemma \ref{liekolchinlemma}, there is some $b \in B_j$ such that $0 \neq b^{-1} \cdot x$ projects nontrivially onto $\g_{\theta_j}$.  As $X_{yw_{0,j}}$ is $B$-invariant and contains $\dot{y}\dot{w}_{0,j}B$, we see that $(b\dot{y}\dot{w}_{0,j})^{-1} \cdot x \in H$.  Now 
$$
(b\dot{y}\dot{w}_{0,j})^{-1} \cdot x = \dot{w}_{0,j}^{-1} \cdot(\dot{y}^{-1} \cdot (b^{-1} \cdot x)).
$$
As $\g_{\theta_j}$ is $B_j$-invariant, we see that $b^{-1} \cdot x$ projects nontrivially onto $\g_{\theta_j}$.   As $y \in (W_j)_{\theta_j}$,  also $\dot{y}^{-1} \cdot (b^{-1} \cdot x)$ projects nontrivially onto $\g_{\theta_j}$.  Finally, $w_{0,j}^{-1}=w_{0,j}$ maps $\theta_j$ to $-\theta_j$ and it follows that $(b\dot{y}\dot{w}_{0,j})^{-1} \cdot x$ projects nontrivially onto $\g_{-\theta_j}$.  

As $H$ projects nontrivially onto the $1$-dimensional root space $\g_{-\theta_j}$, we see that $\g_{-\theta_j} \subseteq H$.  Since $\g_{-\theta_j}$ generates the $B_j$-module $\l_j$ we must have $\l_j \subseteq H$.  In particular, $x_j \in H$, hence $x-x_j \in H$.  As 
$Ad(L_j)$ fixes $x-x_j$, it follows that for all $w \in W_j$,
$$
\dot{w}^{-1} \cdot x=x-x_j+\dot{w}^{-1} \cdot x_j \in H.
$$
In particular, $\dot{w}_{0,j}^{-1} \cdot x \in H$, leading again to the impossibility $\dot{w}_{0,j}B \in X_{yw_{0,j}}$.
\end{proof}

\begin{example}\label{ex.typeA} Consider the Adjoint representation of $SL_n(\C)$ on $\fg = \mathfrak{sl}_n(\C)$.  In this case $\theta = \epsilon_1-\epsilon_n$ and the stabilizer of $\theta$ is the subgroup $W_\theta = \left< s_2, \ldots, s_{n-2} \right>$. Thus $W_\theta \neq \{e\}$ whenever $n\geq 4$ \textup{(}that is, whenever the rank of $\mathfrak{sl}_n(\C)$ is at least $3$\textup{)}.  By Theorem~\ref{thm.not.adjoint}, every element of the right coset $W_\theta w_0$, except for $w_0$, has the property that the corresponding Schubert variety cannot be realized as a Hessenberg variety. Note that the reflection 
\[
s_\theta = s_{n-1}\cdots s_2s_1s_2\cdots s_{n-1} =(1,n)
\] 
is the minimal length coset representative for $W_\theta w_0$. Thus:
\[
W_\theta w_0 \setminus \{w_0\} = W_\theta s_\theta\setminus \{w_0\} = \{ ys_\theta : y\in W_\theta, y\neq y_0 \}
\]
where $y_0$ denotes the longest element of $W_\theta$. This shows $X_{ys_\theta}$ is not equal to any adjoint Hessenberg variety $\B(x,H)$ for all $y\in W_\theta \setminus \{y_0\}$.
\end{example}

Note the set $W_\theta s_\theta \setminus \{w_0\}$ from Example  \ref{ex.typeA} consists of all $w \in S_n$ such that $w_1=n$, $w_n=1$, and there exist $k,\ell\in \{2,\ldots, n-1\}$ such that $k<\ell$ and $w_k<w_\ell$.  In particular, all such $w$ contain the pattern $[4231]$.  The results of Section~\ref{sec.pattern-adjoint} extend Theorem~\ref{thm.not.adjoint} to a general statement about pattern avoidance in the type A case.

To conclude, we note that our proof of Theorem \ref{thm.not.adjoint} can be adjusted easily to obtain the following converse to Theorem~\ref{dem}.

\begin{proposition} \label{notstrict}
Let $\lambda$ be a dominant but not strictly dominant weight for $G$, with associated highest weight representation $\psi:G \rightarrow GL(V(\lambda))$.  Then there exists $w \in W$ such that no pair $(x \in V(\lambda),H \subseteq V(\lambda))$ with $H$ a Hessenberg space satisfies $\B(x,H)=X_w$.
\end{proposition}

%%%%%%%%%%%%%%%%%%%%%%%

\section{The equality question for type A adjoint Hessenberg varieties and pattern avoidance}
\label{sec.pattern-adjoint}

Assume that $G=SL_n(\C)$ throughout. 
In this section we prove Theorem~\ref{pattern.intro}, which says that if $w \in S_n$ contains the pattern $[4231]$, then there is no adjoint Hessenberg variety in $\B$ equal to $X_w$.  
We will see in Example~\ref{ex.no-pattern} that there is no pattern avoidance criterion characterizing the set of type A Schubert varieties that are adjoint Hessenberg varieties.  
Theorem~\ref{pattern.intro} remains true if we take $G=GL_n(\C)$, as can be shown with minor modifications, omitted herein, to our proofs.

As mentioned earlier, Theorem~\ref{pattern.intro} is much more powerful than Theorem~\ref{thm.not.adjoint.intro} when only type A is considered.  Indeed, Theorem~\ref{thm.not.adjoint.intro} implies that the number of $w \in S_n$ such that $X_w$ is not equal to an adjoint Hessenberg variety grows at least quadratically with $n$, while Theorem 1.8 says that the number of such $w$ is at least $n!-c^n$ for some constant $c$.

We begin by recording two basic facts that we will use repeatedly.  Each matrix $x \in \mathfrak{sl}_n(\C)$ can be written as a linear combination of elementary matrices $\{E_{ij} : (i,j) \in [n]\times [n]\}$. We define $\c_{ij}: \mathfrak{sl}_n(\C) \to \C$ to be the coordinate function returning the coefficient of $E_{ij}$ in such a combination.   
Throughout this section, we write $g \cdot x$ for $Ad(g)(x)$ whenever $g \in SL_n(\C)$ and $x \in \mathfrak{sl}_n(\C)$.

\begin{lemma} \label{lem.linearalg} Let $x\in \mathfrak{sl}_n(\C)$ with $\c_{ij}(x)\neq 0$ for some $(i,j)\in [n]\times [n]$ such that $i\neq j$.  If $(k,\ell) \in [n]\times [n]$ with $k\leq i$ and $\ell\geq j$ then there exists $b\in B$ such that $\c_{k,\ell}(b \cdot x) \neq 0$.
\end{lemma}

\begin{proof}
First suppose that $(k,\ell)=(j,i)$. This implies $j<i$ and so $b_\alpha:=I+\alpha E_{ji}\in B$ for all $\alpha\in \C$.
A direct computation shows that $\c_{ji}(b_\alpha\cdot x)=\c_{ji}(x)+\alpha (\c_{ii}(x)-\c_{jj}(x))-\alpha^2\c_{ij}(x)$. Thus, there exists $\alpha \in \C$ such that $\c_{ji}(b_\alpha\cdot x)\neq 0$.
If $i\neq \ell$ the lemma will follow from the existence of $b_1,b_2\in B$ such that $\c_{i\ell}(b_1\cdot x)\neq 0$ and $\c_{k\ell}(b_2\cdot(b_1\cdot x))\neq 0$.
Symmetrically, if $j\neq k$ the lemma will follow from the existence of $b_1,b_2\in B$ such that $\c_{kj}(b_1\cdot x)\neq 0$ and $\c_{k\ell}(b_2\cdot(b_1\cdot x))\neq 0$.
Therefore, to settle the case $(k,\ell)\neq(j,i)$ it suffices to consider the cases $i=k$ and $j=\ell$. 

 When $k=i$ and $\ell=j$ the statement of the Lemma is immediate so let us assume that this is not the case.
If $k=i$ then for each $\alpha \in \C$, $b_\alpha:=I+\alpha E_{j\ell}\in B$ and a direct computation shows that $\c_{i\ell}(b_\alpha\cdot x)=\c_{i\ell}(x)-\alpha \c_{ij}(x)$. Thus, there exists $\alpha \in \C$ such that $\c_{i\ell}(b_\alpha\cdot x)\neq 0$.
Similarly, If $\ell=j$ then $b_\alpha:=I+\alpha E_{ki}\in B$ and $\c_{kj}(b_\alpha\cdot x)=\c_{kj}(x)+\alpha \c_{ij}(x)$, which implies that there exists $\alpha \in \C$ such that $\c_{kj}(b_\alpha\cdot x)\neq 0$.
\end{proof}

Note that Lemma~\ref{lem.linearalg} gives us an explicit description of the partial order $\leq$ on $\Phi$ defined in Section~\ref{sec.demazure}.  Indeed, we have $\fg_{\epsilon_i-\epsilon_j} = \C\{E_{ij}\}$ and the lemma tells us that there exists $b\in B$ such that $\c_{k\ell}(b\cdot E_{ij})\neq0$ so $\fg_{\epsilon_k-\epsilon_\ell}$ is contained in the Demazure module $H_{\epsilon_i-\epsilon_j}$.  In summary, 
\[
\epsilon_i -\epsilon_k \leq \epsilon_k-\epsilon_\ell \Leftrightarrow k\leq i  \ \textup{ and } \ \ell\geq j.
\]  
Note that this description of $\leq$ can also be obtained from Lemma~\ref{lem.partial.order}.

\begin{lemma} \label{lem.elsh}
Say $H \subseteq \mathfrak{sl}_n(\C)$ is a Hessenberg space and $E_{ij} \in H$ with $i \neq j$.  If $k,\ell \in [n]$ with $k \neq \ell$ and both $k \leq i$ and $\ell \geq j$ hold, then $E_{k \ell} \in H$.  In other words, if $\fg_{\beta} \subseteq H$ then $\fg_\gamma\subseteq H$ for all $\gamma\in \Phi$ such that $\beta \leq \gamma$.
\end{lemma}

\begin{proof}
By Lemma~\ref{lem.linearalg} there exists $b\in B$ such that  $\c_{k,\ell}(b \cdot E_{ij}) \neq 0$. 
Since $H$ is $B$-invariant, we have $b\cdot E_{ij}\in H$ implying
$\g_{\epsilon_k-\epsilon_\ell}\subseteq H$ and 
$E_{k\ell}\in H$.
\end{proof}

Theorem \ref{pattern.intro} will follow directly from the next result.

\begin{lemma} \label{extratau}
Assume $w \in S_n$ and $1 \leq i<j<k<\ell \leq n$ with $w_\ell<w_j<w_k<w_i$.  Define $\tau \in S_n$ by
$$
\tau_m:=\left\{ \begin{array}{ll} w_m & m \not \in \{j,k\}, \\ w_k & m=j, \\ w_j & m=k. \end{array} \right.
$$
Let $x \in \fg=\mathfrak{sl}_n(\C)$ and let $H \subseteq \mathfrak{sl}_n(\C)$ be a Hessenberg space.  If the Schubert cell $C_{w^{-1}}$ is contained in the adjoint Hessenberg variety $\B(x,H)$, then $\dot{\tau}^{-1}B \in \B(x,H)$.
\end{lemma}

\begin{proof}
Given $w$, $x$, and $H$ as in the statement, we assume $C_{w^{-1}} \subseteq \B(x,H)$, and write $\c_{pq}$ for $\c_{pq}(x)$ whenever $p,q \in [n]$ with $p \neq q$.  Set 
$$
x^\prime:=\sum_{p \neq q}\c_{pq}E_{pq},
\qquad
\text{and}
\qquad
s:=x-x^\prime\in \h.
$$
Then $x=x'+s$. We will show that both $\dot{\tau}\cdot x^\prime$ and $\dot{\tau}\cdot s$ lie in $H$, thus proving the lemma.

We know that $\dot{w}\cdot x \in H$, since $\dot{w}^{-1}B \in C_{w^{-1}} \subseteq \B(x,H)$.  In particular, $E_{pq} \in H$ whenever $p \neq q$ and $\c_{pq}(\dot{w}\cdot x) \neq 0$.  As $\c_{pq}(\dot{w}\cdot x^\prime)=\c_{pq}(\dot{w}\cdot x)$ whenever $p \neq q$, it follows that $\dot{w}\cdot x^\prime \in H$.  Therefore, $\dot{w}\cdot s=\dot{w} \cdot x-\dot{w}\cdot x^\prime \in H$.

We see now that to show $\dot{\tau} \cdot s \in H$, it suffices to show
\begin{equation} \label{diffinh}
\dot{w}\cdot s-\dot{\tau} \cdot s \in H.
\end{equation}
For $p \in [n-1]$, we write $h_p$ for $E_{pp}-E_{p+1,p+1}\in \h$.  There exist $c_p \in \C$ such that
$$
s=\sum_{p=1}^{n-1} c_ph_p.
$$
We observe that $\dot{w}\cdot E_{mm}=E_{w_m,w_m}$ for all $m \in [n]$, and similar for $\dot{\tau}$.  A direct computation shows that, whether or not $k=j+1$, we obtain
$$
\dot{w}\cdot s-\dot{\tau} \cdot s=(c_{j-1}-c_j-c_{k-1}+c_k)(E_{w_k,w_k}-E_{w_j,w_j}).
$$
If $(c_{j-1}-c_j-c_{k-1}+c_k)=0$, then (\ref{diffinh}) holds.  So, we assume from now on that
\begin{equation} \label{cdiff}
c_{j-1}-c_j-c_{k-1}+c_k \neq 0.
\end{equation}
In order to show that (\ref{diffinh}) holds, we must show that
\begin{equation} \label{2diffinh}
E_{w_k,w_k}-E_{w_j,w_j} \in H.
\end{equation}

We claim that there is some $b \in B$ with $\c_{jk}(b\cdot x) \neq 0$. Given this claim, as $i<j$ and $\ell>k$, Lemma \ref{lem.linearalg} implies that there exists $b^\prime \in B$ such that $\c_{i\ell}(b^\prime b \cdot x) \neq 0$.  It follows that
\begin{equation} \label{bbweq}
\c_{w_iw_\ell}(\dot{w}b^\prime b\cdot x) \neq 0.
\end{equation}
Since $(\dot{w}b^\prime b)^{-1}B \in C_{w^{-1}} \subseteq \B(x,H)$, we see that $\dot{w}b^\prime b \cdot x \in H$.  It follows from (\ref{bbweq}) that $E_{w_iw_\ell} \in H$.  Using Lemma \ref{lem.elsh}, we see that $E_{w_kw_j} \in H$.  Finally, as $E_{w_jw_k} \in \b$ and $H$ is a Hessenberg space, we get 
$$
E_{w_kw_k}-E_{w_jw_j}=-[E_{w_jw_k},E_{w_kw_j}] \in H.
$$
Thus (\ref{2diffinh}) and (\ref{diffinh}) hold.

We aim now to prove the claim that there exists $b \in B$ with $\c_{jk}(b \cdot x) \neq 0$. By Lemma \ref{lem.linearalg}, such $b$ exists if there are $p,q$ such that all of $p \neq q$, $j \leq p$, $q \leq k$, and $\c_{pq} \neq 0$ hold. We may therefore assume no such $p,q$ exist. In that case (see for example \cite[p.61]{Carter}) $\c_{jk}(b\cdot x^\prime)=0$ for all $b \in B$.  Now we choose
$$
b:=I+E_{jk} \in B.
$$
From direct calculation (or by \cite[p.61]{Carter}), we see that
$$
b\cdot s=s-(c_{j-1}-c_j-c_{k-1}+c_k)E_{jk}.
$$
It follows now from (\ref{cdiff}) that
$$
\c_{jk}(b \cdot x)=\c_{jk}(b \cdot x^\prime)+\c_{jk}(b \cdot s) \neq 0
$$
as desired.  This completes our proof that $\dot{\tau} \cdot s \in H$.

It remains to show that $\dot{\tau} \cdot x^\prime \in H$.  Equivalently, we must show that, for all $p,q \in [n]$ with $p \neq q$,
\begin{equation} \label{tauxprime}
\mbox{if } \c_{pq} \neq 0 \mbox{ then } E_{\tau_p\tau_q} \in H.
\end{equation}
Since $\dot{w}^{-1}B \in \B(x,H)$, we know that $E_{w_p,w_q} \in H$ whenever $p \neq q$ and $\c_{pq} \neq 0$.  Therefore, (\ref{tauxprime}) holds whenever $\{p,q\} \cap \{j,k\}=\varnothing$.
For all remaining pairs $(p,q)$, we will choose appropriate pairs $(d,e)$ and $(s,t)$ of indices such that 
	$$
	(w_s,w_t)=(\tau_p,\tau_q),\ \ d\le p,\  \ e\ge q,\ \ w_s\le w_d,\ \text{and}\ w_t\ge w_e.
	$$
Assertion~\eqref{tauxprime} follows in each case from an argument of the following form:
\begin{quote}
``Since $\c_{p,q} \neq 0$, it follows from Lemma \ref{lem.linearalg} that there is some $b \in B$ such that $\c_{d,e}(b\cdot x) \neq 0$.  Therefore, $\c_{w_dw_e}(\dot{w}b\cdot x) \neq 0$.  As $(\dot{w}b)^{-1}B \in C_{w^{-1}} \subseteq \B(x,H)$, we see that $E_{w_dw_e} \in H$.  By Lemma \ref{lem.elsh}, $E_{w_sw_t} \in H$."
\end{quote}

An exhaustive list of pairs $(p,q)$ and corresponding pairs $(d,e)$ and $(s,t)$ is given in the table below.  
$$
\begin{array}{|c|c|c|} \hline (p,q) & (d,e) & (s,t) \\ \hline \hline (j,k) & (i,\ell) & (k,j) \\ \hline (j,i) & (i,j) & (k,i) \\ \hline (j,q), q \not \in \{i,k\} & (i,q) & (k,q) \\ \hline (\ell,j) & (k,\ell) & (\ell,k) \\ \hline (p,j), p \not \in \{k,\ell\} & (p,k) & (p,k) \\ \hline (k,j) & (j,k) & (j,k) \\ \hline (k,q), q \neq j & (j,q) & (j,q) \\ \hline (\ell,k) & (i,\ell) & (\ell,j) \\ \hline \end{array}
$$
We conclude that (\ref{tauxprime}) holds and our proof is complete.
\end{proof}

We can now prove Theorem~\ref{pattern.intro}, restated here for the reader's convenience.

\begin{theorem} \label{pattern}
Let $w \in S_n$.  If some adjoint Hessenberg variety $\B(x,H) \subseteq SL_n(\C)/B$ is equal to the Schubert variety $X_{w}$, then $w$ avoids the pattern $[4231]$.
\end{theorem}

\begin{proof}
It is well-known and straightforward to show that, given permutations $\sigma$ and $w$, $w$ contains the pattern $\sigma$ if and only if $w^{-1}$ contains $\sigma^{-1}$.  Since $[4231]$ is its own inverse, the theorem will follow once we show that if $\B(x,H)=X_{w^{-1}}$, then $w$ avoids the pattern $[4231]$.

If $\B(x,H)=X_{w^{-1}}$ then $C_{w^{-1}} \subseteq \B(x,H)$.  If $w$ contains the pattern $[4231]$ then $\dot{\tau}^{-1}B \in \B(x,H)$, with $\tau$ as in Lemma \ref{extratau}.  We observe that $\ell(\tau)>\ell(w)$, since $\tau$ is obtained from $w$ by exchanging the positions of two letters that appear in ascending order in $w$.  It follows that $\tau \not\leq_\br w$ and $\tau^{-1} \not\leq_\br w^{-1}$.  This forces $\B(x,H) \neq X_{w^{-1}}$.
\end{proof}

As mentioned in the introduction, the Marcus--Tardos Theorem guarantees that the number $\left|S_n(4231)\right|$ of permutations in $S_n$ avoiding $[4231]$ is bounded above by some exponential function of $n$.  In particular,
$$
\lim_{n \rightarrow \infty} \frac{\left|S_n(4231)\right|}{\left|S_n\right|}=0.
$$
So, for large $n$, Schubert varieties in $SL_n(\C)/B$ that are also Hessenberg varieties in $SL_n(\C)/B$ are extremely rare.  To our knowledge, the best exponential bounds for $\left|S_n(4231)\right|$ known currently are (for all large enough $n$)
$$
10.271^n \leq \left| S_n(4231) \right| \leq 13.5^n,
$$
due to Bevan, Brignall, Elvey Price and Pantone in \cite{BBPP}.  It is worth remarking that we do not know if there is any constant $c>1$ such that the number of Schubert varieties in $SL_n(\C)/B$ that are (equal to) Hessenberg varieties in $SL_n(\C)/B$ is at least $c^n$.

\bigskip

We can use similar reasoning as that used in the proof of in Lemma~\ref{extratau} to argue that certain Schubert varieties in $SL_4(\C)/B$ are not equal to any Hessenberg variety in $SL_4(\C)/B$.  These arguments do not extend as easily to pattern avoidance results, but we include one example here for the sake of clarity.  Example~\ref{ex.4patterns} below completely characterizes which Schubert varieties are equal to Hessenberg varieties in $SL_4(\C)/B$.

\begin{lemma}\label{lem.pattern.n=4} If $w\in \{[4123], [2341]\}$ then no adjoint Hessenberg variety is equal to $X_{w^{-1}}$.
\end{lemma}
\begin{proof} First we consider $w=[4123]$.  Note that permutations $\sigma = [2314]$ and $\tau=[1342]$ are the Bruhat-minimal elements of $S_4$ such that $\sigma \nleq w$ and $\tau\nleq w$.  We argue that given $x\in \mathfrak{sl}_4(\C)$ and Hessenberg space $H\subseteq \mathfrak{sl}_4(\C)$ such that  $\B(x,H)\subseteq SL_4(\C)/B$ is a $B$-invariant Hessenberg variety with $X_{w^{-1}}\subseteq \B(x,H)$, then either $\dot\sigma^{-1}B\in \B(x,H)$ or $\dot\tau^{-1}B\in \B(x,H)$.  This proves $X_{w^{-1}}\subsetneq \B(x,H)$ and thus $X_{w^{-1}}$ is not equal to a Hessenberg variety.

To begin, write 
\[
x = \sum_{(p,q)\in [4]\times [4]} \c_{pq}E_{pq}
\]
where $\c_{pq} = \c_{pq}(x)$ for all $p,q$.

If $\{ q\in [4]: \c_{pq} \neq 0 \, \textup{ for some }\, p\neq q \}=\varnothing$ (i.e.~if $x$ is a diagonal matrix) then we set $q_0=5$.  Otherwise, set $q_0:= \min\{ q\in [4]: \c_{pq} \neq 0 \, \textup{ for some }\, p\neq q \}$.    

Suppose first that $q_0=1$ or $q_0=2$.  By Lemma~\ref{lem.linearalg}, there exists $b\in B$ such that $\c_{12}(b\cdot x)\neq 0$ and thus $\c_{41}(\dot w\cdot b\cdot x)\neq 0$.  Since $X_{w^{-1}}\subseteq \B(x,H)$ we know $b^{-1}\dot w^{-1}B\in \B(x,H)$ and therefore $E_{41}\in H$. It follows that $H=\mathfrak{sl}_4(\C)$.  So, $\B(x,H)= SL_4(\C)/B$ contains both $\dot\sigma^{-1}B\in \B(x,H)$ and $\dot\tau^{-1}B\in \B(x,H)$ in this case.

We assume now that $q_0\geq 3$.  This implies, in particular, that 
$$
\c_{21}=\c_{31}=\c_{41}=\c_{12}=\c_{32}=\c_{42}=0,
$$
hence
\begin{equation}\label{eqn.case3}
\quad \  x= \c_{11}E_{11}+\c_{22}E_{22} + \sum_{(p,q)\in [4]\times \{3,4\}} \c_{pq}E_{pq}.
\end{equation}
If $\c_{11}\neq \c_{22}$ then there exists $b\in B$ such that $\c_{12}(b\cdot x)\neq 0$.  (This is easily verified by direct calculation.) Since $\B(x,H)$ is $B$-invariant we see that $\B(x,H) = \B(b\cdot x, H)$.  Thus, if $\c_{11}\neq \c_{22}$ we can argue $\B(x,H) = SL_4(\C)/B$ as above.  So, we assume $\c_{11}=\c_{22}$ for the remainder of the proof.

If $q_0=3$, then $\c_{p3} \neq 0$ for some $p\in \{1,2,4\}$.  By Lemma~\ref{lem.linearalg}, there exists $b\in B$ such that $\c_{13}(b\cdot x) \neq 0$ and therefore $\c_{42}(\dot w\cdot b \cdot x)\neq 0$.  Since $b^{-1}\dot w^{-1}B\in \B(x,H)$ we get $E_{42}\in H$.  Thus: 
\[
\mathfrak{sl}_4(\C)\cap \C\{E_{pq} : p\in [4], q\in \{2,3,4\} \}  \subseteq H.
\]
Since we are working under the assumptions of~\eqref{eqn.case3} above and $\c_{11}=\c_{22}$, we have
\[
\dot\tau\cdot x - x \in \mathfrak{sl}_4(\C)\cap\C\{E_{pq} : p\in [4], q\in \{2,3,4\} \} \subseteq H
\]
and since $x\in H$ (because $eB \in X_{w^{-1}}\subseteq \B(x,H)$), we get $\dot\tau\cdot x\in H$ in this case.

Now assume $q_0\geq 4$.  So, $\c_{pq}=0$ for all $q\in\{1,2,3\}$ and $p\neq q$.  Thus~\eqref{eqn.case3} becomes:
\begin{eqnarray}\label{eqn.case4}
 x= \c_{11}E_{11}+\c_{11}E_{22}+\c_{33}E_{33}+ \sum_{p\in [4]}\c_{p4}E_{p4}
\end{eqnarray}
As above, a direct computation shows that if $\c_{11}\neq \c_{33}$ then there exists $b\in B$ such that $\c_{31}(b\cdot x)\neq 0$.  The arguments of the previous paragraph then imply $\dot\tau^{-1}B\in \B(b\cdot x, H) = \B(x,H)$.  We therefore assume that $\c_{11}=\c_{33}$ for the remainder of the proof.

If $q_0=4$ then $\c_{p4} \neq 0$ for some $p\in \{1,2,3\}$.  By Lemma~\ref{lem.linearalg} there exists $b\in B$ such that $\c_{14}(b\cdot x)\neq 0$ and therefore $\c_{43}(\dot w\cdot b\cdot x) \neq 0$.  This implies $E_{43}\in H$ and thus
\[
\mathfrak{sl}_4(\C)\cap \C\{E_{pq}: p\in [4], q\in \{3,4\}\} \subseteq H.
\]
The assumptions~\eqref{eqn.case4} and $\c_{11}=\c_{33}$ now imply $\dot \sigma \cdot x -x\in H$ and so $\dot\sigma^{-1}B\in \B(x,H)$ in this case.

Finally, if $q_0=5$ then we must have
\[
x = \c_{11}E_{11}+\c_{11}E_{22}+\c_{11}E_{33}+\c_{44}E_{44}.
\]
If $\c_{11}\neq \c_{44}$ then there exists $b\in B$ such that $\c_{14}(b\cdot x) \neq 0$ and the argument of the previous paragraph implies $\dot\sigma^{-1}B\in \B(b\cdot x, H)=\B(x,H)$.  We may therefore assume $\c_{11}=\c_{44}$.  However, this is impossible as $x\in \mathfrak{sl}_4(\C)$.  This concludes the proof for $w=[4123]$.

The proof for $w=[2341]$ follows from similar reasoning using $\sigma = [3124]$ and $\tau=[1423]$. Exchanging the roles of rows and columns in the proof above yields the desired result.  For brevity, we omit the details here.
\end{proof}

\begin{example}\label{ex.4patterns} Let $n=4$. In this case, $X_{w^{-1}}$ is an adjoint Hessenberg variety if and only if 
\[
w \notin \{ [4231], [2341],[4123],[1342],[3124] \}.
\]
Indeed, the first three elements of the set above are exactly the patterns from Theorem~\ref{pattern} and Lemma~\ref{lem.pattern.n=4}.  The proof that $X_{w^{-1}}$ is not equal to an adjoint Hessenberg variety when $w=[1342]$ or $w=[3124]$ is similar to that of Lemma~\ref{lem.pattern.n=4}. Finally, the remaining 19 permutations can each be realized as an adjoint Hessenberg variety.  Indeed, 12 of these permutations have the property that $X_{w^{-1}}$ can be realized as highest weight adjoint Hessenberg varieties by Proposition~\ref{prop.highest-wt-adjoint}.  Of the remaining 12 permutations that are not maximal length coset representatives for $S_4/\left< s_2 \right>$, five appear above and the remaining 7 can be realized as Hessenberg varieties using the data from Table~\ref{figure.table1}.  In this table we denote $H$ as a matrix with starred entries; this means that 
\[
H= \mathfrak{sl}_4(\C) \cap \C\{ E_{ij} :  \textup{ entry $(i,j)$ contains a $*$ } \}.
\]  

\begin{table}[h]
\[
\begin{array}{|c|c|c|c|}\hline
w\in S_4 \textup{ such that } \ell(ws_2)>\ell(w) & \textup{ $X_{w^{-1}} = \B(x, H)$ in $SL_4(\C)/B$? } &  x & H \\   \hline
[4231] & \textup{ no } & - & - \\ \hline 
[3241] & \textup{ yes } & E_{13}+E_{24} & \begin{bmatrix} * & * & * & * \\   * & * & * & * \\   * & * & * & * \\  0 & 0 & 0 & 0 \end{bmatrix} \\ \hline
[2341] & \textup{ no } & - & - \\ \hline 
[1342] & \textup{ no } & - & - \\ \hline 
[3142] & \textup{ yes } & E_{13}+E_{24} & \begin{bmatrix} 0 & * & * & * \\   0 & * & * & * \\   0 & * & * & * \\  0 & 0 & 0 & 0 \end{bmatrix} \\ \hline
[4132] & \textup{ yes } & E_{13}+E_{24} & \begin{bmatrix} 0 & * & * & * \\   0 & * & * & * \\   0 & * & * & * \\  0 & * & * & * \end{bmatrix} \\ \hline
[1243] & \textup{ yes } & E_{12}+E_{24} & \begin{bmatrix} 0 & * & * & * \\   0 & 0 & * & * \\   0 & 0 & 0 & 0 \\  0 & 0 & 0 & 0 \end{bmatrix} \\ \hline
[2143] & \textup{ yes } & E_{13}+E_{24} & \begin{bmatrix} 0 & 0 & * & * \\   0 & 0 & * & * \\   0 & 0 & 0 & 0 \\  0 & 0 & 0 & 0 \end{bmatrix} \\ \hline
[4123] & \textup{ no } & - & - \\ \hline
[3124] & \textup{ no } & - & - \\ \hline
[2134] & \textup{ yes } & E_{13}+E_{34} & \begin{bmatrix} 0 & 0 & * & * \\   0 & 0 & * & * \\   0 & 0 & 0 & * \\  0 & 0 & 0 & 0 \end{bmatrix} \\ \hline
[1234] & \textup{ yes } & E_{12}+E_{24} & \begin{bmatrix} 0 & * & * & * \\   0 & 0 & 0 & * \\   0 & 0 & 0 & 0 \\  0 & 0 & 0 & 0 \end{bmatrix} \\ \hline
\end{array}
\]
\begin{caption}{Answer to the equality question for all Schubert varieties $X_{w^{-1}}$ where $w\in S_4$ is such that $\ell(ws_2)>\ell(w)$.  } 
\label{figure.table1} 
\end{caption}
\end{table}
\end{example}

Given Theorem~\ref{pattern.intro}, it is natural to ask whether the property that $X_w$ is an adjoint Hessenberg variety is characterized by pattern avoidance.  This is not the case, as the next example shows.  
\begin{example}\label{ex.no-pattern} Let $n=5$ and $w=[13425]\in S_5$. A direct computation shows  when 
\[
H = \mathfrak{sl}_5(\C) \cap \C\{ E_{12},E_{13}, E_{14}, E_{15}, E_{25}, E_{35} \}
\]
we get that
$
\B(E_{12}+E_{25},H)=X_{w^{-1}}.
$
Therefore, although $w$ must avoid the pattern $[1342]$ when $n=4$ in order to be realized as an adjoint Hessenberg variety, this is not the case when $n=5$. 
\end{example}

As we have seen in Section~\ref{sec.highestwt} every highest weight Hessenberg variety is $B$-invariant.
The question of whether or not an arbitrary adjoint Hessenberg variety is $B$-invariant is much more nuanced that in the highest weight case.
Indeed, 
Examples~\ref{ex.4patterns} and \ref{ex.no-pattern} exhibit $B$-invariant adjoint Hessenberg varieties that are not highest weight Hessenberg varieties. 
This motivates the following open question. 

\begin{question} Given an adjoint Hessenberg variety $\B(x,H)$ such that $x \notin \C\{E_\theta\}$, what conditions on $x$ and Hessenberg space $H$ guarantee that $\B(x,H)$ is $B$-invariant?  In the case that $\B(x,H)$ is $B$-invariant, what conditions guarantee it is irreducible?
\end{question}

%\newpage
%%%%%%%%%%%%%
\section{The isomorphism question for type A adjoint Hessenberg varieties}\label{sec.typeA}

Our goal is to prove the following result, which is a restatement of Theorem~\ref{adjoint.intro} above.

\begin{theorem} \label{adjoint} Suppose $n\geq 6$.
Assume that $G=GL_n(\C)$ or $G=SL_n(\C)$ and $B$ is the Borel subgroup of $G$ consisting of upper triangular matrices.  Let $w_0$ be the longest element of the Weyl group $W=S_n$, and for $i \in [n-1]$, let $s_i \in W$ be the transposition $(i,i+1)$.  If $3 \leq i \leq n-3$, then there do not exist $x \in \g$ and subspace $H \subseteq \g$ such that $[\b,H] \subseteq H$ and $\B(x,H)$ is isomorphic with $X_{s_iw_0}$.
\end{theorem}

Assume that we have fixed $n \geq 4$.  As each Schubert variety $X_{s_iw_0}$ is irreducible of codimension one in the flag variety $\B$, our plan is to study closely the structure of adjoint Hessenberg varieties of codimension one in $\B$.  In particular, we study the Euler characteristic of irreducible adjoint Hessenberg varieties of codimension one.  We will see that in all but two cases, if $3 \leq i \leq n-3$, no such variety has the same Euler characteristic as $X_{s_iw_0}$.  The remaining two cases will be handled by examining Betti numbers.

A few remarks are in order.  
First, Proposition~\ref{prop.highest-wt-adjoint} tells us that for all $n\leq 3$, every Schubert variety is equal to an adjoint Hessenberg variety and that for all $n$, $X_{s_1w_0}$ and $X_{s_{n-1}w_0}$ are adjoint Hessenberg varieties.
  Second, Theorem~\ref{adjoint} does not resolve the isomorphism question for the Schubert varieties $X_{s_2w_0}$ and $X_{s_{n-2}w_0}$ whenever $n\geq 4$.  
In fact, the existence Hessenberg varieties in $\B$ with the same Betti numbers as $X_{s_2w_0}$ and $X_{s_{n-2}w_0}$ renders the methods used to prove Theorem~\ref{adjoint} useless.  Recall that $\theta=\varepsilon_1-\varepsilon_n$ is the highest root in $\Phi$.  We define the Hessenberg space
$$
H(\overline{-\theta}):=\h \oplus \bigoplus_{\alpha \in \Phi \setminus \{-\theta\}} \g_\alpha.
$$
As is standard, we define the Poincar\'{e} polynomial of a space $X$ as
$$
\textup{Poin}(X,q):=\sum_{i \geq 0}\dim_\C H^i(X;\C)q^i.
$$

\begin{example}\label{ex.Betti} Let $n\geq 4$ and suppose $v=s_2w_0$ or $v=s_{n-2}w_0$.  We then have
\begin{equation} \label{poinpoly}
\textup{Poin}(X_v, \sqrt{q}) = [n-2]_q! \left( [n]_q [n-1]_q - q^{2n-3}-q^{2n-4} \right).
\end{equation}
\textup{(}Here $[n]_q = 1+q+\cdots + q^{n-1}$ and $[n]_q!=\prod_{j=1}^n[j]_q$.\textup{)}  Using the Betti number formulas from either~\cite{Tymoczko2006} or~\cite{Precup2013} along with (\ref{poinpoly}), one can check that if $x=E_{1,n-1}+E_{2,n}$ or $x=E_{1,2}+E_{2,n}$ then $\textup{Poin}(X_v,q) = \textup{Poin}(\B(x, H(\overline{-\theta})),q)$.  One can also confirm this equality for $n\leq 9$ using the tables of Poincar\'e polynomials found on Tymoczko's website~\cite{Tymoczko-web}.
\end{example}

We do not know if either of $X_{s_2w_0}$ and $X_{n-2w_0}$ is isomorphic to either of the Hessenberg varieties appearing in Example~\ref{ex.Betti}.

We now commence our study of codimension one adjoint Hessenberg varieties.  Our main goal is to prove Proposition~\ref{eulercharprop.intro}, restated below for the reader's convenience.

\begin{proposition}
\label{eulercharprop}
If $x \in \g$ and $H \subseteq \g$ is a Hessenberg space such that $\B(x,H)$ is irreducible and has codimension one in $\B$, then the Euler characteristic $\chi(\B(x,H))$ is divisible by $(n-2)!$.
\end{proposition}

Proposition \ref{eulercharprop} will be combined with Lemma \ref{schuec} below to reduce the proof of Theorem~\ref{adjoint} to the examination of two special cases.

\begin{lemma} \label{schuec}
Let $S=\{s_1,\ldots,s_{n-1}\}$ be the set of simple reflections in $S_n$.  For $i \in [n-1]$, let $W(i)$ be the subgroup of $S_n$ generated by $S \setminus \{s_i\}$.  Then
$$
\chi(X_{s_iw_0})=n!-\left|W(i)\right|=n!-i!(n-i)!.
$$
\end{lemma}

\begin{proof}
It follows directly from the definitions that $W(i) \cong S_i \times S_{n-i}$ so $|W(i)| = i!(n-i)!$.  The bijection on $S_n$ sending $w$ to $ww_0$ is an anti-automorphism of the Bruhat order (see for example~\cite[Proposition 2.3.4]{Bjorner-Brenti}).  It follows that the number of Bruhat cells $C_u$ not contained in $X_{s_iw}$ is equal to the number of $u \in S_n$ such that $s_i \not\leq u$, which is the number of elements of $W(i)$. 
\end{proof}

The remainder of this section consists of the proof of Proposition \ref{eulercharprop} followed by the proof of Theorem~\ref{adjoint}.  The key step is the next result, which we will prove after stating a preliminary lemma.  

\begin{lemma}
\label{jtcor}
Assume that $x \in \g$ and that $H \subseteq \g$ is a Hessenberg space.  If $\B(x,H)$ has dimension ${{n} \choose {2}}-1$, then at least one of
\begin{itemize}
\item $H=H(\overline{-\theta})$, or
\item $x$ has an eigenspace of dimension $n-1$
\end{itemize}
must hold.
\end{lemma}

In order to prove Lemma \ref{jtcor}, we define two more ``large" Hessenberg spaces.  Set $\beta_1:=\varepsilon_{n-1}-\varepsilon_1$ and $\beta_2:=\varepsilon_n-\varepsilon_2$.  For $i \in \{1,2\}$ we define the Hessenberg space
$$
H(\overline{\beta}_i):=\h \oplus \bigoplus_{\alpha \in \Phi \setminus\{-\theta,\beta_i\}}\g_\alpha.
$$
Assume temporarily that $G=GL_n(\C)$, $x \in \g$ and $H \subseteq \g$ is a Hessenberg space.  The map $\iota:SL_n(\C)/(B \cap SL_n(\C)) \rightarrow \B$ sending $g(B \cap SL_n(\C))$ to $gB$ is an isomorphism.  Note that $\iota$ factors through the identification of each of $G/B$ and $SL_n(\C)/(B\cap SL_n(\C))$ with the variety of full flags in $\C^n$.  Thus, if $x \in \mathfrak{sl}_n(\C)$ then $\B(x,H)=\iota(\B(x,H \cap \mathfrak{sl}_n(\C)))$, while if $x \not\in \mathfrak{sl}_n(\C)$ and $H \subseteq \mathfrak{sl}_n(\C)$ then $\B(x,H)=\varnothing$.  Therefore, the next assumption is harmless.

\begin{assumption} \label{assumph}
If $G=GL_n(\C)$ and $H \subseteq \g$ is a Hessenberg space, then $H \not\subseteq \mathfrak{sl}_n(\C)$.
\end{assumption}

We continue now under our original assumption that $G=GL_n(\C)$ or $G=SL_n(\C)$.  Given Assumption \ref{assumph}, the next lemma holds for Hessenberg spaces in either $\fg=\mathfrak{sl}_n(\C)$ or $\fg=\mathfrak{gl}_n(\C)$.  

\begin{lemma} \label{smallcodim}
Let $H \subseteq \g$ be a Hessenberg space satisfying Assumption~\ref{assumph}.
\begin{enumerate}
\item If $\dim_{\C} \g/H=1$ then $H=H(\overline{-\theta})$.
\item If $\dim_{\C} \g/H=2$ then $H=H(\overline{\beta_i})$ for some $i \in \{1,2\}$.
\item If $\dim_{\C} \g/H>2$ then $H \subset H(\overline{\beta_i})$ for some $i \in \{1,2\}$.
\end{enumerate}
\end{lemma}

\begin{proof} If $\fg_\gamma \subset H$, then $H$ contains the Demazure module in $\fg$ generated by the root vector $E_\gamma$ so $\fg_\beta \subseteq H$ for all $\beta \geq \gamma$ by~\eqref{eqn.Demazure-order}.  Now if $\pi_{-\theta}(H) \neq 0$ then $\g_{-\theta} \subset H$ and, as $E_{n1} \in \fg_{-\theta}$ generates the Demazure module $\g \cap \mathfrak{sl}_n(\C)$, we have $H=\g$.  Therefore every proper Hessenberg subspace of $\g$ is contained in $H(\overline{-\theta})$ and (1) follows.  Moreover, if $\dim_\C \g/H=2$ then $H \subset H(\overline{-\theta})$.  For all $\gamma \in \Phi\setminus \{-\theta\}$ we have $\gamma\geq \beta_1$ or $\gamma \geq \beta_2$.  Thus $E_{n-1,1} \in \fg_{\beta_1}$ and $E_{n 2} \in \fg_{\beta_2}$ together generate the $B$-module $H(\overline{-\theta})\cap \mathfrak{sl}_n(\C)$, and assertion (2) follows.  The poset of Hessenberg spaces in $\g$ satisfying Assumption~\ref{assumph}, ordered by inclusion, is graded by dimension.   Now (3) follows from~(2). 
\end{proof}

For $x \in \g$ and $c \in \C$, set $$\mult_c(x):=\dim_\C \ker(x-cI).$$ 
Lemma \ref{jtcor} follows directly from Lemma~\ref{smallcodim} and the next result. 
Indeed, if $H\neq H(\overline{-\theta})$ then by Lemma~\ref{smallcodim} we have $H \subseteq H(\overline{\beta_i})$ for some $i \in \{1,2\}$.
Since $\dim \B(x,H)\le \dim \B(x,H(\overline{\beta_i}))$, Proposition~\ref{prop.maxeigenspace} below implies that $\max_{c \in \C}(\mult_c(x))\ge n-1$ whenever $\dim \B(x,H)= {{n} \choose {2}}-1$.
As $\B(x,H)=\varnothing$ or $\B(x,H)=\B$ whenever $x$ is scalar, Lemma \ref{jtcor} follows.

\begin{proposition} \label{prop.maxeigenspace}
For $i \in \{1,2\}$ and $x \in \g$, $\B(x,H(\overline{\beta}_i))$ has dimension at most
$$
\max\left({{n} \choose {2}}-2,{{n-1} \choose {2}}-1+\max_{c \in \C}(\mult_c(x))\right).
$$
\end{proposition}

\begin{proof}
We claim first that it suffices to consider the case $i=1$.  Let $\sigma_0 \in GL_n(\C)$ be the involution mapping $e_i$ to $e_{n+1-i}$ for all $i \in [n]$.  Define an automorphism $r:G \rightarrow G$ by
$$
r(g)=\sigma_0 (g^{-1})^{\mathsf{\mathsf{tr}}} \sigma_0.
$$
We observe that if $b \in B$ then $(b^{-1})^{\mathsf{\mathsf{tr}}}$ is lower triangular, and thus $r(b) \in B$.  Therefore the map $\overline{r}:\B \rightarrow \B$ sending $gB$ to $r(g)B$ is a well defined automorphism of $\B$.  A direct calculation shows that
$$
\overline{r}(\B(x,H(\overline{\beta_2}))=\B(\sigma_0 x^{\mathsf{\mathsf{tr}}} \sigma_0,H(\overline{\beta_1})).
$$
Since $x$ and $\sigma_0 x^{\mathsf{\mathsf{tr}}} \sigma_0$ are conjugate, we see that the dimensions of $\B(x,H(\overline{\beta_1}))$ and $\B(x, H(\overline{\beta_2}))$ are equal. Our claim follows.

Noting that $h(H(\overline{\beta}_1))=(n-2,n,\ldots,n)$, we consider the map $\pi:\B(x,H(\overline{\beta}_1)) \rightarrow \PP^{n-1}$ given by
$$
\pi({\mathcal V}_\bullet)=V_1.
$$
As $\pi$ is the restriction to $\B(x,H(\overline{\beta}_1))$ of the projection of the product of Grassmannians $\prod_{k=1}^{n-1} \grass(k,n)$ onto its first factor, $\pi$ is a regular map.  Let ${\mathcal C}$ be an irreducible component of $\B(x,H(\overline{\beta}_1))$ and let $Y \in \pi({\mathcal C})$.  According to \cite[Corollary 11.13]{Harris}, 
\begin{equation} \label{dimensionbound}
\dim {\mathcal C} \leq \dim \pi ({\mathcal C})+\dim\pi^{-1}(Y).
\end{equation}
The fiber $\pi^{-1}(Y)$ consists of all flags ${\mathcal V}_\bullet$ such that $V_1=Y$ and $Y+xY \leq V_{n-2}$. 
If $Y$ is not $x$-invariant, then $Y+xY$ is two-dimensional.
In this case, we define the (surjective) map $\zeta$ from $\pi^{-1}(Y)$ to the Grassmannian of $(n-4)$-dimensional subspaces of $\C^n/(Y+xY)$ (a projective variety) sending ${\mathcal V}_\bullet$ to $V_{n-2}/(Y+xY)$.
For every such subspace $Z^\prime=Z/(Y+xY)$, $\zeta^{-1}(Z^\prime)$ is isomorphic to the product $\fl_{n-3} \times \PP^1$.  Indeed, $\zeta^{-1}(Z^\prime)$ consists of all ${\mathcal V}_\bullet$ such that $V_1=Y$ and $V_{n-2}=Z$.  Choosing $V_2,\ldots,V_{n-3}$ is equivalent to choosing a flag in $V_{n-2}/V_1$, and choosing $V_{n-1}$ is equivalent to choosing a $1$-dimensional subspace of $\C^n/V_{n-2}$.  It follows from \cite[Theorem 11.14]{Harris} that $\pi^{-1}(Y)$ is irreducible.  We may apply \cite[Theorem 11.12]{Harris} and well-known facts about Grassmannians and flag varieties to get
\begin{equation} \label{dimbound2}
\dim \pi^{-1}(Y)
=
\dim \grass(n-4,n-2) + \dim (\fl_{n-3} \times \PP^1)
=
2(n-4)+{{n-3} \choose {2}}+1.
\end{equation}
Combining (\ref{dimensionbound}) and (\ref{dimbound2}), we get
$$
\dim {\mathcal C} \leq n-1+2(n-4)+{{n-3} \choose {2}}+1={{n} \choose {2}}-2,
$$
and the claim of the Proposition follows in this case.

We are left with the case where every $Y \in \pi({\mathcal C})$ is $x$-invariant.  We consider the map from ${\mathcal C}$ to $\C$ sending ${\mathcal V}_\bullet$ to the eigenvalue of $x$ on $V_1$.  As ${\mathcal C}$ is irreducible and therefore connected, and the set of eigenvalues of $x$ is discrete, there exists some eigenvalue $c$ of $x$ such that every $Y \in \pi({\mathcal C})$ is spanned by an element of $\ker (x-cI)$.
It follows that
\begin{equation} \label{dimbound3}
\dim \pi({\mathcal C}) \leq \mult_c(x)-1.
\end{equation}
For any $Y \in \pi({\mathcal C})$, $\pi^{-1}(Y)$ consists of all flags ${\mathcal V}_\bullet$ such that $V_1=Y$.  It follows that $\pi^{-1}(Y)$ is isomorphic with $\fl_{n-1}$.  Combining (\ref{dimensionbound}) and (\ref{dimbound3}), we get
$$
\dim {\mathcal C} \leq {{n-1} \choose {2}}+\mult_c(x)-1.
$$
This concludes the proof.
\end{proof}

With Lemma \ref{jtcor} in hand, we analyze the two cases arising from its conclusion starting with the case in which $x$ has an eigenspace of dimension $n-1$.

\begin{lemma} \label{largeeigenlem}
Assume that $x \in \g$ has an eigenspace of dimension $n-1$ and $H \subseteq \g$ is a Hessenberg space.  If $\B(x,H)$ is irreducible of codimension one in $\B$, then $\B(x,H)$ is isomorphic with \textup{(}at least\textup{)} one of $X_{s_1w_0}$ or $X_{s_{n-1}w_0}$.
\end{lemma}

\begin{proof}
We observe that if $x$ has an eigenspace of dimension $n-1$, then $x$ is either semisimple or conjugate to $cI+E_{1n}$ for some scalar $c$.  We consider first the case where $x$ is semisimple.  We may assume without loss of generality that there exist distinct constants $c,d$ such that $xe_i=ce_i$ for $i \in [n-1]$, and $xe_n=de_n$.  If $H=\g$, then $\B(x,H)=\B$.  So, we assume that $H \neq \g$.  By Lemma \ref{smallcodim}(1), $H \subseteq H(\overline{-\theta})$ and so $\B(x,H) \subseteq \B(x,H(\overline{-\theta}))$.  We will show that $B(x,H(\overline{-\theta}))$ has two irreducible components, one of which is $X_{s_{n-1}w_0}$ with the other isomorphic to $X_{s_1w_0}$.  The claim of the lemma in this case then follows immediately.

Since $h(H(\overline{-\theta}))=(n-1,n,\ldots,n)$, we see that $\B(x,H(\overline{-\theta}))$ consists of all flags ${\mathcal V}_\bullet$ such that $xV_1<V_{n-1}$.  In particular, if $V_1$ is $x$-invariant, then ${\mathcal V}_\bullet \in \B(x,H(\overline{-\theta}))$.  Define
$$
{\mathcal C}_1:=\{{\mathcal V}_\bullet:V_1 \leq \C\{e_i:i \in [n-1]\}\}.
$$
As $\C\{e_i:i \in [n-1]\}=\ker (x-cI)$, we see that  ${\mathcal C}_1 \subseteq \B(x,H(\overline{-\theta}))$.  By the flag description of Schubert varieties and the fact that $s_{n-1}w_0=[n-1 \, n \, n-2 \cdots 2 \, 1]$, it follows that
$$
{\mathcal C}_1=X_{s_{n-1}w_0}.
$$
Now set
$$
{\mathcal C}_2:=\{{\mathcal V}_\bullet:e_n \in V_{n-1}\}.
$$
We will show that $\B(x,H(\overline{-\theta}))={\mathcal C}_1 \cup {\mathcal C}_2$, and that, with $\sigma_0$ the involution mapping $e_{i}$ to $e_{n+1-i}$ as above, ${\mathcal C}_2=\sigma_0 (X_{s_1w_0})$, thereby completing our examination of the case where $x$ is semisimple.  Let ${\mathcal V}_\bullet \in {\mathcal C}_2$.  If $V_1$ is $x$-invariant, then ${\mathcal V}_\bullet \in \B(x,H(\overline{-\theta}))$.  If $V_1$ is not $x$-invariant, then $V_1=\C\{e_n+y\}$ for some nonzero $y \in \C\{e_i:i \in [n-1]\}$.  Now $xV_1=\C\{de_n+cy\}$.  If $c=0$ then $d \neq 0$ and $xV_1=\C\{e_n\} \leq V_{n-1}$.  If $c \neq 0$ then 
$$V_1+xV_1=\C\left\{e_n+y,e_n+y-\frac{1}{c}\left(de_n+cy\right)\right\}=\C\{e_n+y,e_n\},$$ 
the second equality holding since $d \neq c$.  Now $xV_1<V_1+xV_1<V_{n-1}$.  In either case, ${\mathcal V}_\bullet \in \B(x,H(\overline{-\theta}))$ so $\mathcal{C}_2\subseteq \B(x,H(\overline{-\theta}))$.  

Conversely, say ${\mathcal V}_\bullet \in \B(x,H(\overline{-\theta})) \setminus {\mathcal C}_1$.  If $V_1=\C\{e_n\}$, then $e_n \in V_{n-1}$.  Otherwise, $V_1=\C\{e_n+y\}$ for some nonzero $y \in \C\{e_i:i \in [n-1]\}$.  Arguing as in the paragraph just above, we see that $e_n \in V_1+xV_1$ and therefore $e_n \in V_{n-1}$.  We conclude that ${\mathcal V}_\bullet \in {\mathcal C}_2$ and $\B(x,H(\overline{-\theta})) = \mathcal{C}_1\cup \mathcal{C}_2$ as claimed.

Now $s_1w_0=[n \, n-1 \cdots 3\,1\, 2]$.  By the flag description of Schubert varieties,  $X_{s_1w_0}=\{{\mathcal V}_\bullet: e_1 \in V_{n-1}\}$.   Now
$$
\sigma_0 (X_{s_1w_0})=\sigma_0(\{{\mathcal V}_\bullet:e_1 \in V_{n-1}\})=\{{\mathcal F}_\bullet:e_n \in F_{n-1}\}={\mathcal C}_2,
$$
the second equality following from $\sigma_0e_1=e_n$.

It remains to examine the case where $x$ is conjugate to $cI+E_{1n}$ for some $c \in \C$ (and $H \subseteq \g$ is an arbitrary Hessenberg space).  We claim that in this case, either $\B(x,H)=\varnothing$ or $\B(x,H)=B(E_{1n},H)$.  If $\B(x,H)\neq \varnothing$ then there is some $g \in G$ such that $g^{-1} \cdot x \in H$.  Write $g=b_1\dot{w} b_2$ with $b_1,b_2 \in B$ and $w \in S_n$.  We see that $\dot{w}^{-1} b_1^{-1} \cdot x \in b_2\cdot H=H$.  Now
$$
\dot{w}^{-1} b_1^{-1} \cdot x=cI+\dot{w}^{-1} b_1^{-1} \cdot E_{1n} = cI+d E_{w_1,w_n}
$$ 
for some $d \in \C^\ast$.  In particular, we see that $\pi_{\varepsilon_{w_1}-\varepsilon_{w_n}}(H) \neq 0$, hence $E_{w_1,w_n} \in H$.  It follows that $cI \in H$.  Now, for arbitrary $g \in G$, $g^{-1} \cdot x = cI+g^{-1} \cdot E_{1n}$ lies in $H$ if and only if $g^{-1} \cdot E_{1n} \in H$.   The claim follows.

Now $\B(E_{1n},H)$ is a highest weight Hessenberg variety for the adjoint representation and is therefore a union of Schubert cells.  In particular, if $B(x,H)$ is irreducible and of codimension one, then $\B(x,H)=X_{s_iw_0}$ for some $i \in [n-1]$. We observe that if $1<i<n-1$ then $s_iw_0$ maps $1$ to $n$ and $n$ to $1$ and so lies in the same coset of the stabilizer of $\theta = \varepsilon_1-\varepsilon_n$ as $w_0$.  On the other hand, neither $s_1w_0$ nor $s_{n-1}w_0$ lies in the coset and so both are longest representatives of the coset containing them.  The lemma follows from Proposition \ref{prop.highest-wt-adjoint} and Theorem~\ref{thm.not.adjoint} (and Lemma \ref{schuec}). 
\end{proof}

We record the novel geometric results from the proof of Lemma~\ref{largeeigenlem} below.

\begin{corollary} Let $x$ be a semisimple matrix such that there exists distinct constants $c, d$  with $xe_i=ce_i$ for $i\in [n-1]$ and $xe_n=de_n$. The adjoint Hessenberg variety $\B(x, H(\overline{-\theta}))$ is a union of two irreducible components, one equal to the Schubert variety $X_{s_{n-1}w_0}$ and the other isomorphic to the Schubert variety~$X_{s_1w_0}$.
\end{corollary}

The following lemma, when combined with Lemmas \ref{jtcor} and \ref{largeeigenlem}, will complete the proof of Proposition~\ref{eulercharprop}.

\begin{lemma} \label{jtprop}
For every $x \in \g$, the Euler characteristic $\chi(\B(x,H(\overline{-\theta})))$ is divisible by $(n-2)!$.
\end{lemma}

\begin{proof}
We assume without loss of generality that $x=x_s+x_n$ with $x_s$ diagonal and $x_n$ an upper triangular, nilpotent matrix satisfying the assumptions of~\cite[Corollary 4.9]{Precup2013}.  Let ${\mathcal C}(x,H(\overline{-\theta}))$ be the set of all $w \in W$ such that the Schubert cell $C_w$ has nonempty intersection with $\B(x,H(\overline{-\theta}))$.  By \cite[Theorem 5.4]{Precup2013} and its proof, the affine spaces $\B(x,H(\overline{-\theta})) \cap C_w$ with $w \in {\mathcal C}(x,H)$ determine an affine paving of $\B(x,H(\overline{-\theta}))$.  So,
$$
\chi(\B(x,H(\overline{-\theta}))) =\dim_\C H^\ast(\B(x,H(\overline{-\theta})))=|{\mathcal C}(x,H(\overline{-\theta}))|.
$$
Now by \cite[Proposition 3.7]{Precup2013}, $w \in {\mathcal C}(x,H(\overline{-\theta}))$ if and only if $Ad(\dot{w}^{-1})(x_n) \in H(\overline{-\theta})$. We now write
$$
x_n=\sum_{j<k} c_{jk} E_{j,k}
$$
for some $c_{jk}\in \C$.
It follows that ${\mathcal C}(x,H(\overline{-\theta}))$ fails to contain exactly those $w \in W$ such that $w(-\theta)=\varepsilon_j-\varepsilon_{k}$ for pairs $j<k$ satisfying $c_{jk} \neq 0$.  In particular, $W \setminus {\mathcal C}(x,H(\overline{-\theta}))$ is a union of cosets of the stabilizer of $-\theta$ in $W$.  As this stabilizer is isomorphic with $S_{n-2}$, the lemma follows.
\end{proof}

We are now ready to prove Theorem~\ref{adjoint}.

\begin{proof}[Proof of Theorem~\ref{adjoint}.]
Observe that $X_{s_iw_0}$ is irreducible.  So, if there is some adjoint Hessenberg variety $\B(x,H) \subseteq \B$ isomorphic with $X_{s_iw_0}$, then $i!(n-i)!$ is divisible by $(n-2)!$ by Proposition~\ref{eulercharprop} and Lemma~\ref{schuec}.  However, if $n \geq 9$ and $3 \leq i \leq n-3$, then $0<i!(n-i)!<(n-2)!$.  Moreover, inspection shows that if $6 \leq n \leq 8$ and $3 \leq i \leq n-3$, then $(n-2)!$ does not divide $i!(n-i)!$ unless $(n,i) \in \{(8,3),(8,5)\}$.

We assume now that $n=8$ and $i \in \{3,5\}$.  In this case, we have $\chi(X_{s_iw_0}) = 8! - 3!5!$.  If $\B(x,H)$ is irreducible of codimension one in $\B$ then we have by Lemmas~\ref{jtcor} and~\ref{largeeigenlem} that $H=H(\overline{-\theta})$.  Let $x=x_s+x_n$ be the decomposition of $x$ as in the proof Lemma~\ref{largeeigenlem}.  Arguing as in that proof, we get that 
\[
\chi(\B(x,H(\overline{-\theta}))) = |\mathcal{C}(x, H(\overline{-\theta}))| = 8! - 6! \,|\{(j<k) : c_{jk}\neq 0\}|.
\]
It follows immediately that if $\chi(\B(x,H(\overline{-\theta}))) = 8!-3!5!$, then $x_n=E_{jk}$ for some $1\leq j<k\leq n$.  Finally, one can use the formulas given in~\cite[Corollary 5.5]{Precup2013} to check that there is no Hessenberg variety $\B(x_s+E_{jk}, H(\overline{-\theta}))$ in the flag variety $\B=SL_{8}(\C)/B$ with Betti numbers equal to those of the Schubert varieties $X_{s_3w_0}$ and $X_{s_5w_0}$.  For the sake of brevity, we omit these computations.
\end{proof}

%%%%%%%%%%%%%%%%%%%%%%%%%%%%%%%%%%%%%%%%%%%%%
%%%%%%%%%%%%%%%%%%%%%%%%%%%%%%%%%%%%%%%%%%%%%
\section{Type C adjoint Hessenberg varieties}\label{sec.typeC}

Recall from the introduction that the type C flag variety can be identified as the fixed-point set of a type A flag variety under a certain automorphism $\sigma$. Similarly, the Type A and Type C Schubert varieties are closely connected as each type C Schubert variety is the variety of $\sigma$-fixed points of a type A Schubert variety (see~\cite[Chapter 6]{Lakshmibai-Raghavan}).  The first main result of this section is that the same is true of type C Hessenberg varieties. Namely Theorem~\ref{thm: C to A} below says that every type C adjoint Hessenberg variety is the variety of $\sigma$-fixed points of a type A Hessenberg variety and implies Theorem~\ref{foldhess}.  With this groundwork in place, we establish the type C pattern avoidance result stated in Theorem~\ref{patternc.intro} (see Theorem~\ref{patternc} below).  We remark that many of our proofs in this section would be considerably easier and shorter were we to assume that every Hessenberg space $H$ contains the Borel subalgebra $\b$.

We begin by fixing the notation needed to define the automorphism $\sigma$.
Let $E$ be the $2n\times 2n$ block matrix 
\[
E = \begin{bmatrix} 0& J \\ -J & 0 \end{bmatrix}
\]
where $J$ is the $n\times n$ matrix with $1$'s on the anti-diagonal, and $0$'s elsewhere.  
We follow \cite[Chapter 6]{Lakshmibai-Raghavan} and identify $Sp_{2n}(\C)$ with the fixed point set of the involution $\sigma: SL_{2n}(\C) \to SL_{2n}(\C)$ defined by $\sigma(A) = E (A^{\mathsf{\mathsf{tr}}})^{-1}E^{-1}$. 
Explicitly, consider the embedding $\phi:Sp_{2n}(\C) \hookrightarrow SL_{2n}(\C)$ whose image stabilizes the alternating form $\langle -,- \rangle$ defined by 
$$
\langle e_i,e_j \rangle=\left\{ \begin{array}{cc} 1 & j=2n+1-i, \\ 0 & \mbox{otherwise}. \end{array} \right.
$$ 
(Here $e_1,\ldots,e_{2n}$ is the standard basis of $\C^{2n}$.)  
Then $G_C:=\phi(Sp_{2n}(\C))$ is the group of $\sigma$-fixed points in $SL_{2n}(\C)$. Throughout this section we identify $Sp_{2n}(\C)$ with $G_C$.

The maximal torus $T$ in $SL_{2n}(\C)$ consisting of diagonal matrices and fixed Borel subgroup $B$ in $SL_{2n}(\C)$ consisting of upper triangular matrices are stable under $\sigma$, and $T^\sigma$ (respectively, $B_C:=B^\sigma$) is a maximal torus (respectively, Borel subgroup) in $G_C$.  
For $i \in [2n]$, set
$$i'=2n+1-i.
$$
Let $W$ be the Weyl group of $Sp_{2n}(\C)$.  The embedding $\phi$ induces an embedding $\phi^*: W \to S_{2n}$ with image $W_C:=\phi^*(W)$ consisting of those $w\in S_{2n}$ satisfying $w(i)' = w(i')$ for all $i \in [2n]$.  We call such $w$ \textbf{signed permutations}.

We will write $\sigma$ for the differential $d\sigma$ which is the involution of the Lie algebra given by
\[
\sigma : \mathfrak{sl}_{2n}(\C) \to \mathfrak{sl}_{2n}(\C), \; \sigma(x) = Ex^{\mathsf{\mathsf{tr}}} E.
\]
We identify $\mathfrak{sp}_{2n}(\C)$ with $\fg_C:=\mathfrak{sl}_{2n}(\C)^\sigma$.  We observe that $\h^\sigma$ is the Lie algebra of $T^\sigma$, and 
\begin{eqnarray}\label{eqn.CSA}
\h^\sigma = \{\textup{diag}(d_1, \ldots, d_{2n}) \in \mathfrak{sl}_{2n}(\C) : d_i = -d_{i'} \}.
\end{eqnarray}

The involution $\sigma$ also induces an involution of $\h^*$ defined by
\[
{\sigma}: \h^* \to  \h^*,\;\; \sigma(\epsilon_i) = -\epsilon_{i'}.
\]
We describe now a surjective map from the type A root system $\Phi_A$ to the type C root system $\Phi_C$ known as the \textbf{folding map}. 
Set $\bar{\epsilon}_i:= \epsilon_i-\epsilon_{i'}$ and note that, by definition, $\bar{\epsilon}_i = -\bar{\epsilon}_{i'}$. The folding map is now defined to be
\begin{eqnarray}\label{eqn.folding}
\varphi: \Phi_A \to \Phi_C,\;\; \varphi(\epsilon_i-\epsilon_j)= \frac{1}{2}(\bar{\epsilon_i} - \bar{\epsilon_j}).  
\end{eqnarray}

Let $H_C \subseteq \fg_C$ be a type C Hessenberg space. 
Our first goal is to construct a type A Hessenberg space $H$ whose $\sigma$-fixed points are the elements of $H_C$.
Define $\Phi_{H_C}:= \{\gamma\in \Phi_C : \fg_\gamma \subseteq H_C\}$ and  set
\begin{eqnarray}\label{eqn.Hroots}
\Phi_H: = \{\epsilon_i-\epsilon_j \in \Phi_A : \varphi(\epsilon_i-\epsilon_j) \in \Phi_{H_C}\} \subseteq \Phi_A.
\end{eqnarray}
With $H \cap \h$ to be described below, we define $H \subseteq \mathfrak{sl}_{2n}(\C)$ to be the subspace such that 
\begin{eqnarray}\label{eqn.Hess.C.to.A}
H = (H\cap \h) \oplus \bigoplus_{\epsilon_i-\epsilon_j\in \Phi_H} \C\{E_{ij}\}.
\end{eqnarray}
Given $i,j \in [2n]$, we set
$$
h_{ij}:=[E_{ij}, E_{ji}] = E_{ii}-E_{jj},
$$
and define
\[
H\cap \h:= \C\{h, h_{ij}: h\in H_C\cap \h^\sigma \, \textup{ and } \, i,j\in[2n] \, \textup{ such that } \, \epsilon_i - \epsilon_j,\epsilon_j -\epsilon_i \in \Phi_H \}.
\]

We observe that the action of $\sigma$ on $SL_{2n}(\C)$ induces an automorphism of the type A flag variety $\B_A:=SL_{2n}(\C)/B$, which will also be denoted by $\sigma$.  The map $\phi^\prime$ from $\B_C:=G_C/B_C$ to $\B_A$ sending $gB_C$ to $gB$ is a well-defined embedding, and $(\B_A)^\sigma=\phi^\prime(\B_C)$ (see, for example, \cite[Proposition 6.1.1.1]{Lakshmibai-Raghavan}).

We can now state the main theorem of this section.

\begin{theorem}\label{thm: C to A}
Given a type $C$ Hessenberg space $H_C \subseteq \fg_C$, let $H$ be the subspace of $\mathfrak{sl}_{2n}(\C)$ defined as in~\eqref{eqn.Hess.C.to.A} above. 
	\begin{enumerate}
	\item The subspace $H$ is a type A Hessenberg space such that $H^\sigma = H_C$.
	\item Let $x\in \fg_C$. The image under $\phi^\prime$ of the type C Hessenberg variety $\B_C(x,H_C)$ is $\B_A(x, H)^\sigma$.
	\end{enumerate}
\end{theorem}

We prove first that in Theorem~\ref{thm: C to A}, (1) implies (2).
The proof of Theorem~\ref{thm: C to A}\textup{(}1\textup{)} is delayed until after Lemma~\ref{lemma.upper.order}.
 
\begin{proof}[Proof of Theorem~\ref{thm: C to A}\textup{(}2\textup{)}.] 
We wish to show that $\phi^\prime(\B_C(x,H_C)) = \B_A(x,H)^\sigma$ as varieties.
The set $\mathcal{G}_C(x,H_C):=\{g\in G_C: g^{-1}\cdot x \in H_C\}$ is a $B_C$-invariant subvariety of $G_C$ whose image under the morphism $\mu:G_C\to \B_C$ is $\B_C(x,H_C)$.
Let $\mathcal{G}_A(x,H):=\{g\in SL_{2n}(\C): g^{-1}\cdot x \in H\}$.
Since $G_C=SL_{2n}(\C)^\sigma$, we have that $\mathcal{G}_A(x,H)^\sigma=\mathcal{G}_A(x,H)\cap G_C$ is a subvariety of $G_C$.

Let us show that $\mathcal{G}_A(x,H)^\sigma=\mathcal{G}_C(x,H_C)$, as subvarieties of $G_C$.
By (1) of Theorem~\ref{thm: C to A} we have
	$$
	\mathcal{G}_C(x,H_C)=\{g\in G_C: g^{-1}\cdot x \in H,\ g^{-1}\cdot x\in \mathfrak{sl}_{2n}(\C)^\sigma \}.
	$$
Given $g\in G_C$, since $x\in \fg_C$, we have $g^{-1}\cdot x \in \fg_C$ and thus $g^{-1}\cdot x$ is $\sigma$-stable.
It follows that the constraints imposed by $g^{-1}\cdot x\in \mathfrak{sl}_{2n}(\C)^\sigma$ are redundant, so indeed $\mathcal{G}_C(x,H_C)=\mathcal{G}_A(x,H)^\sigma$.
In particular, 
	$$
	\phi'(\B_C(x,H_C))=
	\phi'(\mu(\mathcal{G}_C(x,H_C)))=
	\phi'(\mu(\mathcal{G}_A(x,H)^\sigma))
	$$
as subvarieties of $\B$.
Finally, note that by definition of the maps,
	\begin{align*}
	\phi'(\mu(\mathcal{G}_A(x,H)^\sigma))	&=
	\phi'(\{gB_C\in\B_C	:	\sigma(g)=g,\ g^{-1}\cdot x\in H\})	\\&=
	\{gB\in\B	:	\sigma(g)=g,\ g^{-1}\cdot x\in H\}	\\&=
	\B_A(x,H)^\sigma.
	\end{align*}
This concludes the proof.
\end{proof}

Now we develop the tools needed to prove Theorem~\ref{thm: C to A}(1).  Recall the partial order $\leq$ on the root system $\Phi$ defined in Section~\ref{sec.demazure} above, and the concrete description of that order given in Lemma~\ref{lem.partial.order}.
Our next result tells us that the folding map interacts nicely with $\sigma$, is compatible with $\le$, and is well-behaved with respect to the $W_C$-action.  A proof can be found in \cite[Chapter 6.1]{Lakshmibai-Raghavan}.

\begin{lemma}\label{lemma.roots} Let $\varphi: \Phi_A \to \Phi_C$ be the folding map defined as in~\eqref{eqn.folding}.  This map satisfies each of the following conditions.
\begin{enumerate}
\item $\varphi(\Phi_A^+) = \Phi_C^+$. 
\item Given $\gamma\in \Phi_C$, $\varphi^{-1}(\gamma)$ is precisely the $\sigma$-orbit of any $\gamma'\in \Phi_A$ such that $\varphi(\gamma') = \gamma$.
\item The map $\varphi$ is compatible with the partial ordering $\leq$ on $\Phi_A$ and $\Phi_C$, that is, given $\gamma_1, \gamma_2\in \Phi_A$ we have $\gamma_1 \leq \gamma_2$ implies $\varphi(\gamma_1) \leq \varphi(\gamma_2)$.  
\item $\varphi$ is equivariant with respect to the canonical action of $W_C$ on $\Phi_A$ and $\Phi_C$.
\end{enumerate}
\end{lemma}

For use below, we recall that there is a simple description of the partial order $\leq$ on the root system $\Phi_A$, given by
\begin{eqnarray}\label{eqn.partial.order}
\epsilon_i - \epsilon_j \leq \epsilon_k - \epsilon_\ell \Leftrightarrow k\leq i \textup{ and } \ell \geq j.
\end{eqnarray}

Consider the surjective linear map 
\[
\bar{\sigma}: \mathfrak{sl}_{2n}(\C)\to \fg_C,\;\; \bar{\sigma}(x) = x+\sigma(x).
\] 
Note that $\bar{\sigma}$ is not a Lie algebra homomorphism. 
However, the next lemma tells us that $\bar{\sigma}$ maps the root spaces of $\mathfrak{sl}_{2n}(\C)$ onto those of $\mathfrak{sp}_{2n}(\C)$.

\begin{lemma}\label{lemma.C-to-A-action} Let $h\in \h^\sigma$.  For all $1\leq k,\ell\leq 2n$ with $k\neq \ell$, we have $\bar{\sigma}(E_{k\ell}) \in \fg_{ \varphi(\epsilon_k - \epsilon_\ell)}$ and $\varphi(\epsilon_k - \epsilon_\ell)(h) =  (\epsilon_k - \epsilon_\ell)(h)$.
\end{lemma}
\begin{proof} Since $E_{k\ell}$ is a root vector corresponding to the root $\epsilon_k-\epsilon_\ell\in \Phi_A$ we have $[h, E_{k\ell}] = (\epsilon_k - \epsilon_\ell)(h)E_{k\ell}$. As $\sigma$ is an involution of $\mathfrak{sl}_{2n}(\C)$ and $\sigma(h)=h$ we get
\[
\sigma([h, \sigma(E_{k\ell})]) = [h, E_{k\ell}] = (\epsilon_k-\epsilon_\ell)(h)E_{k\ell}\Rightarrow [h, \sigma(E_{k\ell})] =  (\epsilon_k-\epsilon_\ell)(h) \sigma(E_{k\ell}).
\]
Thus,
\begin{eqnarray*}
[h, \bar{\sigma}(E_{k\ell})] &=&  [h, E_{k\ell}]+[h, \sigma(E_{k\ell})] \\
&=& (\epsilon_k-\epsilon_\ell)(h) E_{k\ell}+ (\epsilon_k-\epsilon_\ell)(h) \sigma(E_{k\ell})\\
&=&(\epsilon_k-\epsilon_\ell)(h) \bar{\sigma}(E_{k\ell}).
\end{eqnarray*}
To conclude the argument, we have only to show that $\varphi(\epsilon_k - \epsilon_\ell)(h) = (\epsilon_k - \epsilon_\ell)(h)$.  Since $h\in \h^\sigma$, we have $\epsilon_k(h) = -\epsilon_{k'}(h)$ for all $k\in [2n]$ (see \eqref{eqn.CSA}).  Thus 
\[
\varphi(\epsilon_k - \epsilon_\ell)(h) = \frac{1}{2}(\bar{\epsilon}_k - \bar{\epsilon}_\ell)(h) = \frac{1}{2}(\epsilon_k - \epsilon_{k'} - \epsilon_\ell + \epsilon_{\ell'})(h) = (\epsilon_k - \epsilon_\ell)(h)
\]
as desired.
\end{proof}

Before arguing that $H$ is a type A Hessenberg space, we prove it is $\sigma$-invariant.

\begin{lemma}\label{lemma: H sigma stable} Let $H_C \subseteq \mathfrak{sp}_{2n}(\C)$ be a type C Hessenberg space and $H \subseteq \mathfrak{sl}_{2n}(\C)$ be the subspace defined as in~\eqref{eqn.Hess.C.to.A} above. Then $\sigma(H)= H$.
\end{lemma}
\begin{proof} It suffices to show that $\sigma(H)\subseteq H$.  To do this, we check that 
\begin{enumerate}
\item $\sigma(E_{ij}) \in H$ for all $i,j$ such that $\epsilon_i-\epsilon_j\in \Phi_H$, i.e., such that $\varphi(\epsilon_i - \epsilon_j) \in \Phi_{H_C}$,
\item $\sigma(h_{ij})\in H$ for all $i,j$ such that $\epsilon_i-\epsilon_j,\epsilon_j-\epsilon_i \in \Phi_H$, and
\item $\sigma(h)\in H$ for all $h\in H_C\cap \h^\sigma$.
\end{enumerate}
Condition (3) follows immediately from the facts that $h\in \h^\sigma$ and $H_C\cap \h^\sigma\subseteq H$.  Note that $\sigma(E_{ij}) \in \C\{E_{j'i'}\}$, so to prove (1) it suffices to show that $\epsilon_{j'}-\epsilon_{i'}\in \Phi_H$, i.e.~that $\varphi(\epsilon_{j'}-\epsilon_{i'}) \in \Phi_{H_C}$. But since
\begin{eqnarray}\label{eqn.folding2}
\varphi(\epsilon_{j'}-\epsilon_{i'}) = \frac{1}{2}(\bar{\epsilon}_{j'} - \bar{\epsilon}_{i'}) = \frac{1}{2}( \bar{\epsilon}_{i} - \bar{\epsilon}_j) = \varphi(\epsilon_i-\epsilon_j)
\end{eqnarray}
we do indeed get $\varphi(\epsilon_{j'}-\epsilon_{i'}) = \varphi(\epsilon_i-\epsilon_j) \in \Phi_{H_C}$.

Finally, we prove (2). Suppose $i,j\in [2n]$ such that $\pm(\epsilon_i-\epsilon_j)\in \Phi_H$, and consider 
\[
\sigma(h_{ij}) = \sigma(E_{ii}-E_{jj}) =  - E_{i'i'} + E_{j'j'} = h_{j'i'}.
\]
By~\eqref{eqn.folding2} we have $\varphi(\epsilon_{j'}-\epsilon_{i'})=\varphi(\epsilon_i-\epsilon_j) \in \Phi_{H_C}$ so $\epsilon_{j'}-\epsilon_{i'} \in \Phi_H$ and similarly $\epsilon_{i'}-\epsilon_{j'}\in \Phi_H$ also.
By the definition of $H$, this implies $h_{j'i'}\in H$, as desired.
\end{proof}

Next, we argue that $\Phi_H$ is an upper order ideal with respect to the partial order $\leq$ on $\Phi_A$.

\begin{lemma} \label{lemma.upper.order} Let $H_C \subseteq \fg_C$ be a type C Hessenberg space. The corresponding subset $\Phi_H\subset\Phi_A$ defined as in~\eqref{eqn.Hroots} is an upper order ideal with respect to the partial order $\leq$ on $\Phi_A$. 
\end{lemma}
\begin{proof} Suppose $i,j\in [2n]$ such that $\epsilon_i-\epsilon_j \in \Phi_H$ and let $\epsilon_k - \epsilon_\ell\in \Phi_A$ with $\epsilon_k - \epsilon_\ell \geq \epsilon_i - \epsilon_j$.  By Lemma~\ref{lemma.roots}, $\varphi(\epsilon_k - \epsilon_\ell) \geq \varphi(\epsilon_i- \epsilon_j)$.
Since $H_C$ is $B$-invariant, $\Phi_{H_C}$ is an upper order ideal with respect to the partial order $\leq$ on $\Phi_C$ (see Section~\ref{sec.demazure}).  Thus as  $\varphi(\epsilon_i - \epsilon_j) \in \Phi_{H_C}$, we have $\varphi(\epsilon_k - \epsilon_\ell)\in \Phi_{H_C}$ and so $\epsilon_k - \epsilon_\ell\in \Phi_H$ as desired.
\end{proof}

\vspace*{.1cm}

\begin{proof}[Proof of Theorem~\ref{thm: C to A}(1).] We begin by arguing that $[\fb, H]\subseteq H$, which proves  that $H$ is a type A Hessenberg space.  Since $\fb = \h \oplus \C\{E_{k\ell} : 1\leq k<\ell \leq 2n\}$, we have only to show that 
\begin{itemize}
\item[(A)] $[h, x] \in H$ for all $h\in \h$ and $x\in H$, and that 
\item[(B)] $[E_{k\ell}, x]\in H$ for all $1\leq k<\ell\leq 2n$ and $x\in H$.  
\end{itemize}
By the definition of $H$ we may write
\[
x = h' + \sum_{ \substack{1\leq i<j \leq n\\ \pm (\epsilon_i - \epsilon_j)\in \Phi_H}} d_{ij} h_{ij}+ \sum_{\epsilon_i-\epsilon_j\in \Phi_H} c_{ij}E_{ij}
\]
for $c_{ij}, d_{ij}\in \C$ and $h'\in H_C\cap \h^\sigma$. 
Assertion (A) admits a straightforward proof: since $h'+ \sum_{ \pm (\epsilon_i - \epsilon_j)\in \Phi_H } d_{ij} h_{ij}\in  \h$ we have
\[
[h, x] = \sum_{\epsilon_i-\epsilon_j \in \Phi_H}c_{ij} [h, E_{ij}] = \sum_{\epsilon_i-\epsilon_j \in \Phi_H} c_{ij} (\epsilon_i-\epsilon_j)(h) E_{ij} \in H
\]
by the definition of $H$.

We turn to assertion (B). Let $k, \ell\in [2n]$ with $k<\ell$. 
We have
\begin{eqnarray} \label{eklbracket}
[E_{k\ell}, x] = [E_{k\ell}, h'] + \sum_{ \substack{1\leq i<j \leq n \\ \pm (\epsilon_i - \epsilon_j)\in \Phi_H}} d_{ij}[ E_{k\ell}, h_{ij}]  +   \sum_{\epsilon_i-\epsilon_j\in \Phi_H} c_{ij}[ E_{k\ell}, E_{ij}].
\end{eqnarray}
To prove $[E_{k\ell}, x]\in H$ we argue that every Lie bracket appearing in each summand on the right side of (\ref{eklbracket}) is an element of $H$.

\vspace*{.1in}

\textbf{Case 1:} Suppose $i\neq j$ with $\epsilon_i-\epsilon_j\in \Phi_H$.  First if $\{i,j\} = \{k, \ell\}$ we have
\[
[E_{k\ell}, E_{ij}] = \left\{\begin{array}{cc}  0 & \textup{ if $i= k$ and $j = \ell$} \\ 
-h_{ij} & \textup{ if $i=\ell$ and $j=k$.} \end{array} \right.
\]
Thus if $[E_{k\ell},E_{ij}] \neq 0$, then $j=k<\ell = i$. In this case our assumption that $\epsilon_i-\epsilon_j\in \Phi_H$ implies $\epsilon_j-\epsilon_i\in \Phi_H$ since $\Phi_H$ is an upper-order ideal by Lemma~\ref{lemma.upper.order}. Now $h_{ij}=[E_{ij},E_{ji}]\in H$, by definition of $H\cap \h$. 
On the other hand, if $\{i,j\} \neq \{k, \ell\}$ and $[E_{k\ell}, E_{ij}]\neq 0$ then $[E_{k\ell}, E_{ij}] \in \fg_{\gamma}$ where $\gamma= (\epsilon_k - \epsilon_\ell ) + (\epsilon_i- \epsilon_j)$.  In particular, we have $ \epsilon_i -\epsilon_j \leq \gamma$ since $\epsilon_k - \epsilon_\ell \in \Phi_A^+$.  This implies $\gamma\in \Phi_H$ by Lemma~\ref{lemma.upper.order}, hence $[E_{k\ell}, E_{ij}] \in \fg_{\gamma} \subseteq H$.

\vspace*{.1in}

\textbf{Case 2:} Let $i,j\in [2n]$ such that $i<j$ and $\epsilon_i-\epsilon_j, \epsilon_j-\epsilon_i \in \Phi_H$.  By definition we must have $h_{ij} = E_{ii}-E_{jj} \in H\cap \h$.  Furthermore, we know $[E_{k\ell}, h_{ij}] = (\epsilon_\ell - \epsilon_k)(h_{ij}) E_{k\ell}$.  Thus, we have only to show that $\epsilon_k -  \epsilon_\ell \in \Phi_H$ whenever $(\epsilon_\ell - \epsilon_k)(h_{ij})\neq 0$. Since
\[
(\epsilon_\ell - \epsilon_k)(h_{ij}) = \delta_{\ell i }-\delta_{\ell j } -\delta_{ki}+\delta_{kj},
\]
where $\delta_{ab}$ is the Kronecker delta function, the condition $(\epsilon_\ell - \epsilon_k)(h_{ij}) \neq 0$ implies $\{i,j\} \cap \{k,\ell\} \neq \varnothing$. Now, we consider the various possibilities and show that that $k\leq j$ and $\ell \geq i$ in each case.
Recall that $i<j$ and $k < \ell$ by assumption.
\begin{itemize}
\item If $i=k$ then $k=i<j$ and $\ell >k=i$.

\item If $i=\ell$, then it follows immediately that $k< \ell = i <j$. Therefore, $k<j$ and $\ell \geq i$.

\item If $i\neq k$ and $i\neq \ell$ then our assumptions imply that either $j=k$ or $j=\ell$.  In the first case, we get $i<j=k<\ell$, hence $k\leq j$ and $\ell >i$.  In the second case, we get $k<\ell = j$ and $\ell = j >i$.
\end{itemize}
Since $k\leq j$ and $\ell \geq i$, we see that $\epsilon_j-\epsilon_i\le \epsilon_k-\epsilon_\ell$.
The assumption that $\epsilon_j-\epsilon_i\in \Phi_H$ implies $\epsilon_k - \epsilon_\ell \in \Phi_H$ by Lemma~\ref{lemma.upper.order}, as desired.

\vspace*{.1in}

\textbf{Case 3:} Suppose $h'\in H_C \cap \h^\sigma$.  As above, we get $[E_{k\ell}, h'] = (\epsilon_\ell-\epsilon_k)(h') E_{k\ell}$ and thus $[E_{k\ell}, h'] \in H$ will follow if we are able to show $\epsilon_k - \epsilon_\ell \in \Phi_H$ whenever $(\epsilon_\ell-\epsilon_k)(h')\neq 0$.  By Lemma~\ref{lemma.C-to-A-action}, we have 
\[
[\bar{\sigma}(E_{k\ell}), h'] = -\varphi(\epsilon_k - \epsilon_\ell)(h') \bar{\sigma}(E_{k\ell}) = (\epsilon_\ell-\epsilon_k)(h') \bar{\sigma}(E_{k\ell}).
\]
We observe that $H_C$ is a type C Hessenberg space, $\bar{\sigma}(E_{k\ell})\in\mathfrak{b}^\sigma$, and $h'\in H_C$. It follows that $[\bar{\sigma}(E_{k\ell}), h] \in H_C$.
Hence, the assumption $(\epsilon_\ell-\epsilon_k)(h')\neq 0$ implies $\bar{\sigma}(E_{k,\ell})\in H_C$.  However, since $\bar{\sigma}(E_{k\ell})$ spans the root space corresponding to $\varphi(\epsilon_k - \epsilon_\ell)$ in $\fg_C$ we have $\varphi(\epsilon_k - \epsilon_\ell)\in \Phi_{H_C}$ and thus $\epsilon_k- \epsilon_\ell \in \Phi_H$ as desired.

\vspace*{.1in}
Having settled all cases, we conclude that $H$ is a type A Hessenberg space.
To complete the proof, we establish now that $H^\sigma = H_C$.  Since, by Lemma~\ref{lemma: H sigma stable}, $H$ is $\sigma$-stable we know $H^\sigma = \bar{\sigma}(H)$ and we will show $\bar{\sigma}(H) = H_C$. 
  Since $H_C\subseteq H$ by definition, the inclusion $H_C\subseteq \bar{\sigma}(H)$ is a consequence of $\bar{\sigma}(H_C)=H_C$.

We verify that $\bar{\sigma}(H) \subseteq H_C$.
It is straightforward that $\bar \sigma(h)\in H_C$ if $h\in H_C\cap \mathfrak h^\sigma$.
Next, if $\epsilon_i - \epsilon_j\in \Phi_H$ then $\varphi(\epsilon_i-\epsilon_j)\in \Phi_{H_C}$ and  Lemma~\ref{lemma.C-to-A-action} now implies $\bar{\sigma}(E_{ij}) \in H_C$.  
Last, consider $1\leq i < j \leq 2n$ such that $\epsilon_i - \epsilon_j,\epsilon_j - \epsilon_i \in \Phi_H$. Then $\varphi(\epsilon_i- \epsilon_j),-\varphi(\epsilon_i- \epsilon_j)  \in \Phi_{H_C}$ and by another application of Lemma~\ref{lemma.C-to-A-action} we get
\[
\bar{\sigma}(E_{ij})\in \fb^\sigma, \bar{\sigma}(E_{ji})\in H_C \Rightarrow [\bar{\sigma}(E_{ij}), \bar{\sigma}(E_{ji})] \in H_C.
\]
Using the fact that $\sigma(E_{ji}) \in \C\{E_{i'j'}\}$ and $i=j'$ if and only if $i'=j$ we obtain  
	\[
	[E_{ij}, \sigma(E_{ji})] = \begin{cases} 0 & \textup{ if $j \neq i'$} \\
	 h_{ij} & \textup{ if $j=i'$.}
	 \end{cases}
	\]
Therefore we have
\begin{eqnarray*}
[\bar{\sigma}(E_{ij}), \bar{\sigma}(E_{ji})] &=& [E_{ij}, E_{ji}] + [\sigma(E_{ij}), E_{ji}]+[E_{ij}, \sigma(E_{ji})]+ [\sigma(E_{ij}), \sigma(E_{ji})]\\
&=& [E_{ij}, E_{ji}]+\sigma([E_{ij}, E_{ji}]) + [E_{ij}, \sigma(E_{ji})]+\sigma([E_{ij}, \sigma(E_{ji})])\\
&=& \bar{\sigma}([E_{ij}, E_{ji}]) + \bar{\sigma}([E_{ij}, \sigma(E_{ji})]) \\
&=& c\, \bar{\sigma}(h_{ij})
\end{eqnarray*}
with $c\in \{1,2\}$. This implies $\bar{\sigma}(h_{ij})\in H_C$, and we conclude that $\bar{\sigma}(H) \subseteq H_C$.
\end{proof}

%%%%%%%%%%%%%%%%%%%%%%%%%%%%%%%%%%%%%%%%%%%%%
\subsection{Type C Pattern Avoidance}

Recall that $W_C$ denotes the subgroup in $S_{2n}$ consisting of signed permutations.  Thus the notion of pattern avoidance as defined in Section~\ref{sec.typeA-setup} makes sense for elements of $W_C$.  
The objective of this section is to prove Theorem~\ref{patternc.intro} using Theorem~\ref{patternc} below, generalizing the pattern avoidance result of Theorem~\ref{pattern} to the type C setting.

We may write each matrix in $\fg_C$ in terms of the Chevalley basis of type C, fixed as in \cite[Section 6.1]{Lakshmibai-Raghavan}.  Given $\gamma\in \Phi_C$, we denote by $E_\gamma $ the corresponding root vector and by $\c_\gamma: \mathfrak{sp}_{2n}(\C) \to \C$ the coordinate function returning the coefficient of $E_\gamma$.  Note that if $\epsilon_i-\epsilon_j \in \Phi_A$ such that $\varphi(\epsilon_i-\epsilon_j) = \gamma$ then $\C\{E_\gamma\} = \C\{ \bar{\sigma}(E_{ij})\}$ by Lemma~\ref{lemma.C-to-A-action}. In particular, for all $y\in \fg_C$ we have 
\begin{eqnarray}\label{eqn.proj}
\c_{\gamma}(y)\neq0
	\Leftrightarrow
	\c_{ij}(y)\neq 0 \Leftrightarrow \c_{j'i'}(y)\neq 0
\end{eqnarray}
where $\c_{ij}: \mathfrak{sl}_{2n}(\C) \to \C$ denotes the coordinate function returning the coefficient of $E_{ij}$ as in Section~\ref{sec.pattern-adjoint} above.
We begin with the following statement, which is a Type C analogue of Lemma~\ref{lem.linearalg}.

\begin{lemma}\label{lemma.typeC-Baction} Let $x\in \fg_C$ and suppose $\c_\gamma (x)\neq 0$ for some $\gamma\in \Phi_C$.  For all $\beta\in \Phi_C$ such that $\beta \geq \gamma$ there exists $b\in B_C$ such that $\c_\beta (b \cdot x) \neq 0$.
\end{lemma}

\begin{proof} 
Let $\gamma', \beta'\in \Phi_A$ be such that $\varphi(\gamma')=\gamma$ and $\varphi(\beta')=\beta$. By definition, $\gamma'\in \{\epsilon_i-\epsilon_j, \epsilon_{j'}-\epsilon_{i'}\}$ and $\beta' \in \{\epsilon_k-\epsilon_\ell , \epsilon_{\ell'}-\epsilon_{k'} \}$ for some $i,j,k,l\in [2n]$ such that $i\neq j$ and $k\neq \ell$.  Since $B_C \subseteq B$, by definition of the partial order $\leq$, we may assume without loss of generality that $\gamma' = \epsilon_i-\epsilon_j$ and $\beta'= \epsilon_k - \epsilon_\ell$ with $\beta'\geq \gamma'$, i.e. $k\le i$ and $\ell\ge j$.
By~\eqref{eqn.proj}, our assumption that $\c_\gamma(x)\neq 0$ implies $\c_{ij}(x)\neq 0$ and to complete the proof of the lemma, it suffices to verify that there exists $b\in B_C$ such that $\c_{k\ell}(b\cdot x)\neq 0$.  Throughout the rest of the proof, $\alpha$ will be a parameter that can take any complex value we choose.

If $(k,\ell)=(j,i)$ then $j<i$ and so $b_\alpha=I+\alpha E_{ji}+\alpha\sigma(E_{ji})\in B_C$.  We have $b_\alpha^{-1}=I-\alpha E_{ji}-\alpha\sigma(E_{ji})$ and $\sigma(E_{ji}) = \pm E_{i'j'}$.  Note that if $\sigma(E_{ji}) = -E_{i'j'}$, then we must have either $i,j\leq n$ or $i,j>n$ so $\delta_{ij'} = 0$ in that case.  Using the fact that 
the $(k,\ell)$-entry of the product of three $n\times n$ matrices $X,Y,Z$ is $\sum_{p=1}^n\sum_{q=1}^n X_{kp}Y_{pq}Z_{q\ell}$ we obtain
	$$
	\c_{i\ell}(b_\alpha\cdot x)=\c_{ji}(x)+\alpha(1+\delta_{ij'})(\c_{ii}(x)-\c_{jj}(x))+\alpha^2(1+3\delta_{ij'})(-\c_{ij}(x)).
	$$
Since $(1+3\delta_{ij'})(-\c_{ij}(x))\neq0$, there exists $\alpha \in \C$ such that $\c_{ji}(b_\alpha\cdot x)\neq 0$.

If $i\neq \ell$ the lemma will follow from the existence of $b_1,b_2\in B_C$ such that $\c_{i\ell}(b_1\cdot x)\neq 0$ and $\c_{k\ell}(b_2\cdot(b_1\cdot x))\neq 0$.
Symmetrically, if $j\neq k$ the lemma will follow from the existence of $b_1,b_2\in B_C$ such that $\c_{kj}(b_1\cdot x)\neq 0$ and $\c_{k\ell}(b_2\cdot(b_1\cdot x))\neq 0$.
Therefore, to settle the case $(k,\ell)\neq(j,i)$ it suffices to consider $i=k$ or $j=\ell$. 
Since both follow from similar arguments, we only write the proof assuming $k=i$. Moreover, we assume $\ell\neq j$ since the case $(i,j)=(k,\ell)$ is trivial.

If $\ell=j'$ then $b=I+\alpha E_{j\ell}\in B_C$.
Since $\c_{i\ell}(b\cdot x)=\c_{i\ell}(x)-\alpha \c_{ij}(x)$ is a nonzero polynomial in $\C[\alpha]$ we can choose $\alpha$ such that $\c_{i\ell}(b\cdot x)\neq 0$.
If $\ell\neq j'$ define $b\in B_C$ by
	$$
	b:=
	\begin{cases}
	I+\alpha E_{j\ell}+\alpha E_{\ell' j'}, & |\{j,\ell\}\cap[n]|= 1\\
	I+\alpha E_{j\ell}-\alpha E_{\ell' j'}, & |\{j,\ell\}\cap[n]|\neq 1
	\end{cases},
	\quad\text{so}
	\quad
	b^{-1}=
	\begin{cases}
	I-\alpha E_{j\ell}-\alpha E_{\ell' j'}, & |\{j,\ell\}\cap[n]|= 1\\
	I-\alpha E_{j\ell}+\alpha E_{\ell' j'}, & |\{j,\ell\}\cap[n]|\neq 1
	\end{cases}.
	$$
Using the formula for the $(k,\ell)$-entry of the product of three $n\times n$ matrices once more, we obtain,
	$$
\c_{i\ell}(b \cdot x)=
\begin{cases}
\c_{i\ell}(x)-\alpha \c_{ij}(x)+ \alpha\delta_{i\ell'}\c_{j'\ell}(x)-\alpha^2\delta_{i\ell'}\c_{j'j}(x), & |\{j,\ell\}\cap[n]|= 1\\
\c_{i\ell}(x)-\alpha \c_{ij}(x)- \alpha\delta_{i\ell'}\c_{j'\ell}(x)+\alpha^2\delta_{i\ell'}\c_{j'j}(x), & |\{j,\ell\}\cap[n]|\neq 1.
\end{cases}
	$$
If $i\neq \ell'$ then it is immediate that we can choose $\alpha$ such that $\c_{i\ell}(b^{-1}\cdot x)\neq 0$.
Finally, let's suppose that $i=\ell'$ and note that $ |\{j,\ell\}\cap[n]|= 1$ if and only if $ |\{j,i\}\cap[n]|\neq 1$. Since $x\in \fg_C$ it follows that the coefficient of $\alpha$ in $\c_{i\ell}(b\cdot x)$ is $-2\c_{ij}(x)\neq 0$. 
Since $\c_{i\ell}(b\cdot x)$ is a nonzero polynomial in $\C[\alpha]$ we can choose $\alpha$ such that $\c_{i\ell}(b\cdot x)\neq 0$.
\end{proof}

\begin{proposition}\label{prop: cell contained C then A} Assume that $x\in \fg_C$ and that $\B_C(x, H_C) \subseteq G_C/B_C$ is a type C Hessenberg variety.  Let $w \in W_C$ with $B_C\dot wB_C \subseteq \B_C(x, H_C)$.  If $H\subseteq \mathfrak{sl}_{2n}(\C)$ is the type A Hessenberg space defined using $H_C$ as in~\eqref{eqn.Hess.C.to.A} above,
then the type A Schubert cell $B\dot wB$ is contained in $ \B(x, H)$.
\end{proposition}

\begin{proof} 
We write
\[
x = h + \sum_{\gamma\in \Phi_{C}} d_{\gamma} E_\gamma
\]
where $d_\gamma\in \C$, $E_{\gamma}$ is a nonzero root vector in $\fg_\gamma \subseteq \fg_C$, and $h\in \h_C$.  We set $x_1:= \sum_{\gamma\in \Phi_{C}} d_{\gamma} E_\gamma$. So, $x=h+x_1$.  Note that $\c_\gamma(x) = d_\gamma$ for all $\gamma\in \Phi_{C}$.

To prove $B\dot wB \subseteq \B_A(x,H)$, it suffices to show that $u\dot wB\in \B_A(x,H)$ for all $u$ in the unipotent radical $U$ of $B$.  
In particular, we must show
\begin{eqnarray}\label{eqn.adjoint1}
(u\dot w)^{-1}\cdot x = \dot w^{-1}u^{-1}\cdot h + \dot w^{-1}u^{-1}\cdot x_1 \in H.
\end{eqnarray}
Since $U$ is unipotent, the exponential map $\exp: \mathfrak{u} \to U$ is a diffeomorphism.  Therefore we may write $u = \exp(y)$ for some  $y = \sum_{ p < q } \c_{pq}(y) E_{pq} \in \mathfrak{u}$.

Now we compute $u^{-1}\cdot x_1$ and $u^{-1}\cdot h$. By properties of the adjoint representation we obtain
\begin{eqnarray*}
u^{-1}\cdot x_1 &=& \Ad(\exp(-y))(x_1) = \exp(\ad_{-y})(x_1) \\
&=& x_1 + \sum_{m=1}^\infty \frac{1}{m!} \ad_{-y}^m (x_1)  
= x_1 + \sum_{m=1}^\infty \sum_{\gamma\in \Phi_{C}} \frac{1}{m!}d_\gamma \ad_{-y}^m (E_\gamma). 
\end{eqnarray*}
Given $\gamma\in \Phi_C$ such that $d_\gamma\neq 0$, we write $\varphi^{-1}(\gamma) = \{\epsilon_i-\epsilon_j, \epsilon_{j'}-\epsilon_{i'}\}$ for some $i,j\in [2n]$ with $i\neq j$.  By Lemma~\ref{lemma.C-to-A-action}, we know that  $E_\gamma\in \C\{E_{ij}, E_{j'i'}\}$ and thus for all $m\geq 1$ we have
\[
\ad_{-y}^m (E_\gamma) \in \C\{E_{k\ell}:  k \leq i, \ell \geq j \,\textup{ or }\, k \leq j', \ell \geq i'  \} \oplus \C\{h_{k\ell} : j \leq k <\ell \leq i  \, \textup{ or } \, i' \leq k <\ell \leq j'  \}.
\]
(As above, $h_{k\ell}=[E_{k\ell},E_{\ell k}]=E_{kk}-E_{\ell\ell}$.) In particular, we see that
\begin{eqnarray}\label{eqn.x1-support}
u^{-1}\cdot x_1 \in \bigoplus_{d_\gamma\neq 0} \, \bigoplus_{\substack{\epsilon_i-\epsilon_j\in \Phi_A\\ \varphi(\epsilon_i-\epsilon_j)=\gamma}} \left( \C\{E_{k\ell} :  k \leq i, \ell \geq j \} \oplus \C\{h_{k\ell} : j \leq k <\ell \leq i \} \right).
\end{eqnarray}
Next, we have that 
\begin{eqnarray*}
u^{-1}\cdot h &=& \Ad(\exp(-y))(h) = \exp(\ad_{-y})(h) 
= h + \sum_{m=1}^\infty \frac{1}{m!} \ad_{-y}^m (h)\\
&=& h + \sum_{m=1}^\infty \, \sum_{p<q } \frac{1}{m!}\c_{pq}(y) (\epsilon_p-\epsilon_q)(h) \ad_{-y}^{m-1}(E_{pq}).
\end{eqnarray*}
Applying similar reasoning as above and using the fact that $p<q$ we have 
\begin{eqnarray}\label{eqn.h-support}
u^{-1}\cdot h - h \in \bigoplus_{\substack{ \epsilon_p-\epsilon_q \in \Phi_A^+\\ (\epsilon_p-\epsilon_q)(h)\neq 0 }} \C\{E_{k\ell} : k\leq p, \ell \geq q  \}.
\end{eqnarray}
Note that $\dot w^{-1}\cdot h\in H_C \subseteq H$ by assumption.  Thus equations~\eqref{eqn.adjoint1},~\eqref{eqn.x1-support} and~\eqref{eqn.h-support} imply that to prove the proposition, it suffices to show 
\begin{eqnarray}\label{eqn.x1-support2}
\quad \quad \bigoplus_{d_\gamma\neq 0} \, \bigoplus_{\substack{\epsilon_i-\epsilon_j\in \Phi_A\\ \varphi(\epsilon_i-\epsilon_j)=\gamma}} \left( \C\left\{E_{w^{-1}_k w^{-1}_\ell} :  k \leq i, \ell \geq j \right\} \oplus \C\left\{h_{w^{-1}_kw^{-1}_\ell} : j \leq k <\ell \leq i \right\} \right)\subseteq H
\end{eqnarray}
and
\begin{eqnarray}\label{eqn.h-support2}
\bigoplus_{\substack{ \epsilon_p-\epsilon_q \in \Phi_A\\ (\epsilon_p-\epsilon_q)(h)\neq 0 }} \C\left\{E_{w^{-1}_kw^{-1}_\ell} : k\leq p, \ell \geq q  \right\} \subseteq H.
\end{eqnarray}

First we establish~\eqref{eqn.x1-support2}.  If $d_\gamma\neq 0$, Lemma~\ref{lemma.typeC-Baction} implies that for each $\beta\in \Phi_C$ with $\beta \ge \gamma$ there exists $b\in B_C$ such that,
	$$
	0\neq \c_\beta(b\cdot x)=\c_{w^{-1}(\beta)}(\dot w^{-1}b\cdot x).
	$$
Our assumption that $B_C\dot w B_C\subseteq \B_C(x, H_C)$ now implies $E_{w^{-1}(\beta)}\in H_C$ and thus $w^{-1}(\beta)\in \Phi_{H_C}$ for all $\beta \geq \gamma$.  Note that by Lemma~\ref{lemma.roots}, for all $\epsilon_k- \epsilon_\ell \in \Phi_A$ with $\epsilon_k-\epsilon_\ell\geq \epsilon_i-\epsilon_j$ we have $\varphi(\epsilon_k-\epsilon_\ell) \geq \gamma$.
This implies $w^{-1}(\varphi(\epsilon_k-\epsilon_\ell)) \in \Phi_{H_C}$ and, since $\varphi$ is $W_C$-equivariant, we have $\varphi(\epsilon_{w^{-1}_k} - \epsilon_{w^{-1}_\ell}) \in \Phi_{H_C}$.  Using the description of $\leq$ for $\Phi_A$ from~\eqref{eqn.partial.order} we have now proved:
\begin{eqnarray*} \label{eqn.typeAroots}
d_{\varphi(\epsilon_i-\epsilon_j)} \neq 0 &\Rightarrow& \epsilon_{w^{-1}(k)} - \epsilon_{w^{-1}(\ell)}  \in \Phi_H \; \textup{ for all } \; k\leq i \; \textup{ and } \; \ell \geq j\\
&\Rightarrow& E_{w^{-1}_kw^{-1}_\ell} \in H \; \textup{ for all } \; k\leq i \; \textup{ and } \; \ell \geq j.
\end{eqnarray*}
It follows that
\begin{eqnarray*}\label{eqn.typeAroots.CSA}
d_{\varphi(\epsilon_i-\epsilon_j)} \neq 0 \textup{ and } i>j &\Rightarrow& 
 E_{w^{-1}_k w^{-1}_\ell}, E_{w^{-1}_\ell, w^{-1}_k} \in H \; \textup{ for all $k,\ell$ such that } \; j \leq k < \ell \leq i \nonumber \\
&\Rightarrow& h_{w^{-1}_kw^{-1}_\ell} \in H \; \textup{ for all $k,\ell$ such that } \; j \leq k <\ell \leq i,
\end{eqnarray*}
where the last implication follows from the fact that $[\fb, H]\subseteq H$.  This concludes the proof of~\eqref{eqn.x1-support2}.

Next we prove~\eqref{eqn.h-support2}.  Fix $\epsilon_p-\epsilon_q \in \Phi_A$ such that $(\epsilon_p-\epsilon_q)(h)\neq 0$.  First, we note that if $\varphi(\epsilon_p-\epsilon_q ) \geq \gamma$ for some $\gamma\in \Phi_C$ such that $d_\gamma\neq 0$, then 
\[
\C\left\{E_{w^{-1}_kw^{-1}_\ell} : k\leq p, \ell \geq q  \right\} \subseteq H
\]
by~\eqref{eqn.x1-support2}.  Thus it suffices to consider the case in which $\varphi(\epsilon_p-\epsilon_q ) \not\geq \gamma$ for any $\gamma \in \Phi_C$ such that $d_\gamma\neq 0$.  
This last assumption implies $E_{pq}$ is not a summand of $u^{-1} \cdot x_1$ for any $u \in U$, i.e., $\c_{pq}(u^{-1}\cdot x_1)=0$ for all $u\in U$.
Consider $u_{pq} := I_n + \bar\sigma(E_{pq}) \in B_C$. Applying Lemma~\ref{lemma.C-to-A-action} and using properties of the adjoint action we have
\begin{eqnarray*}
u_{pq}^{-1}\cdot h =  h - [ \bar{\sigma}(E_{pq}), h] + \sum_{i=2}^{\infty} \frac{1}{i!}\ad_{-\bar\sigma(E_{pq})}^i(h)
=  h + (\epsilon_p - \epsilon_q)(h) \bar{\sigma}(E_{pq}).
\end{eqnarray*}
This implies
\begin{eqnarray*}
u_{pq}^{-1}\cdot x &=& u_{pq}^{-1}\cdot h + u_{pq}^{-1}\cdot x_1 
=  h + (\epsilon_p - \epsilon_q)(h) \bar{\sigma}(E_{pq}) + u_{k\ell}^{-1}\cdot x_1.
\end{eqnarray*}
Since $\bar\sigma(E_{pq}) \in \fg_{\varphi(\epsilon_p - \epsilon_q)}$ and $\c_{pq}(u_{pq}^{-1}\cdot x_1)=0$, it follows that $\c_{\varphi(\epsilon_p - \epsilon_q)}(u_{pq}^{-1}\cdot x) = (\epsilon_p - \epsilon_q)(h)\neq 0$.  
Since $u_{pq}\in B_C$, the assumption $B_C\dot wB_C \subseteq \B_C(x, H_C)$ implies $\dot w^{-1}bu_{k\ell}^{-1}\cdot x \in H_C$  for all $b\in B_C$.
Furthermore, Lemma~\ref{lemma.typeC-Baction} implies that for each $\beta \ge \varphi(\epsilon_p-\epsilon_q)$ there exists $b\in B_C$ such that
	$$
	0\neq \c_\beta(bu_{pq}^{-1}\cdot x)=\c_{w^{-1}(\beta)}(\dot w^{-1}bu_{pq}^{-1}\cdot x)
	$$
and thus we have $w^{-1}(\beta)\in \Phi_{H_C}$ for all $\beta\geq \varphi(\epsilon_p-\epsilon_q)$.  In particular, arguing as in the proof of~\eqref{eqn.x1-support2} we have
\[
E_{w^{-1}_k w^{-1}_\ell}\in H \, \textup{ for all } \, k \leq p \textup{ and } \ell \geq q.
\]
This establishes~\eqref{eqn.h-support2} and completes the proof.
\end{proof}

We conclude this section with the proof of Theorem~\ref{patternc.intro}.

\begin{theorem}\label{patternc}
Let $G=Sp_{2n}(\C)$ and let $B \leq G$ be the Borel subgroup whose image under $\phi$ consists of upper triangular matrices.  Fix $w \in W$.  If there exist $x \in \g=\mathfrak{sp}_{2n}(\C)$ and a Hessenberg space $H_C \subseteq \g$ such that $\B_C(x,H_C)=X_{w^{-1}}^C$, then $\phi^\ast(w)$ avoids the pattern $[4231]$.
\end{theorem}

We remark that since $[4231]$ is its own inverse, this theorem implies that if $\B_C(x,H_C)=X_{w}^C$, then $\phi^\ast(w)$ avoids the pattern $[4231]$.
So Theorem~\ref{patternc} implies Theorem~\ref{patternc.intro}.

\begin{proof}
Suppose $w \in W_C$ contains the pattern $[4231]$.
Seeking a contradiction, suppose there exists $x\in \fg_C$ and a type $C$ Hessenberg space $H_C\subseteq \fg_C$ such that $X_{w^{-1}}^C= \B_C(x,H_C)$.
By Proposition~\ref{prop: cell contained C then A}, $B\dot w^{-1}B\subseteq \B(x, H)$ where $H\subseteq \mathfrak{sl}_{2n}(\C)$ is the type A Hessenberg space defined using $H_C$ as in~\eqref{eqn.Hess.C.to.A} above.
To obtain a contradiction, we show that there exists $v\in W_C$ such that $v\nleq_{\Bru}  w$ and $\dot v^{-1} B\in \B_A(x, H)$.
Given this statement, we would then have by Theorem~\ref{thm: C to A} that $\dot v^{-1} B\in(\B_A(x, H))^\sigma=\B_C(x,H_C)$ contradicting our assumption that $X_{w^{-1}}^C= \B_C(x,H_C)$.

Since $w$ contains the pattern $[4231]$ there exist $i,j,k,\ell$ such that $1\leq i<j<k<\ell\leq 2n$ and $w_\ell<w_j<w_k<w_i$. Consider $\tau=(w_j,w_k)w$ where $(w_j,w_k)$ is the transposition exchanging $w_j$ and $w_k$, so $\tau\in S_{2n}$ is as defined in the statement of Lemma~\ref{extratau}. 
By Lemma~\ref{extratau}, $\dot \tau^{-1}B\in \B(x,H)$, that is, $\dot \tau^{-1}\cdot x\in H$.
If $j'=k$ then $\tau\in W_C$ and taking $v=\tau$ accomplishes the desired goal.
We may therefore assume $j'\neq k$ for the remainder of the proof. 

We have that $1\leq \ell'<k'<j'<i'\leq 2n$ and $w_{i'}<w_{k'}<w_{j'}<w_{\ell'}$.
Let $v=(w_{j'},w_{k'}) \tau = (w_{j'},w_{k'})(w_j,w_k)w$, where $(w_{j'},w_{k'})$ is the transposition exchanging $w_{j'}$ and $w_{k'}$.   Note that $v\nleq_{\Bru}  w$ since $\ell(v)>\ell(w)$ and $v\in W_C$.
In order to argue that $\dot v\cdot x\in H$, we write $x=s+x'$ where 
	$$
	s=\sum_{p\in[2n-1]}c_p(E_{pp}-E_{p+1,p+1})
	\qquad
	\text{and}
	\qquad
	x'=\sum_{\substack{(p,q)\in[2n]\times[2n]\\ p\neq q}}\c_{pq}(x)E_{pq}
	$$
and show $\dot v\cdot s,\dot v\cdot x'\in H$. Note that $\dot w \cdot s, \dot w \cdot x\in H$ by assumption.
From the proof of Lemma~\ref{extratau} we have that $\dot \tau\cdot s\in H$ and
	$$
	\dot w\cdot s-\dot \tau\cdot s =  (c_{j-1}-c_j-c_{k-1}+c_k)(E_{w_kw_k}-E_{w_jw_j})\in H.
	$$
Since $\sigma(H)\subset H$, we see that
	\begin{align*}
	\sigma(\dot w\cdot s-\dot \tau\cdot s) 
	&=
	(c_{j-1}-c_j-c_{k-1}+c_k)(-E_{w_{k'}w_{k'}}+E_{w_{j'}w_{j'}})
	\\&=
	(c_{k'-1}-c_{k'}-c_{j'-1}+c_{j'})(E_{w_{j'}w_{j'}}-E_{w_{k'}w_{k'}})
	\in H,
	\end{align*}
where the last equality follows from the fact that $s\in \h_C$ (see~\eqref{eqn.CSA}).
A direct computation shows that 
\begin{eqnarray*}
	\dot w\cdot s-\dot v\cdot s &=&  (c_{j-1}-c_j-c_{k-1}+c_k)(E_{w_kw_k}-E_{w_jw_j})\\
	&&\quad\quad+(c_{k'-1}-c_{k'}-c_{j'-1}+c_{j'})(E_{w_{j'}w_{j'}}-E_{w_{k'}w_{k'}}),
\end{eqnarray*}
which lies in $H$, and therefore $\dot w\cdot s\in H$ implies $\dot v\cdot s\in H$.

Next we show that $\dot v\cdot x'\in H$ by proving that $E_{v_pv_q}\in H$ whenever $\c_{pq}(x)\neq0$. Let $p\neq q$ be such that $\c_{pq}(x)\neq 0$.  Let $\tau' = (w_{j'},w_{k'})w$, and note that we have $\dot\tau' \cdot x \in H$ by Lemma~\ref{extratau}.  In fact, by the proof of Lemma~\ref{extratau} applied to both $\tau$ and $\tau'$ we have
\[
\{p,q\}\cap \{j',k'\} = \varnothing \Rightarrow E_{v_pv_q} = E_{\tau_p\tau_q} \in H
\]
and
\[
\{p,q\} \cap \{j,k\} = \varnothing  \Rightarrow E_{v_pv_q} = E_{\tau'_p \tau'_q} \in H.
\]
To complete the proof, suppose that $\{p,q\}\cap\{j',k'\}\neq\varnothing$ and $\{p,q\}\cap\{j,k\}\neq \varnothing$. 
If $p\in\{j',k'\}$ and $q\in\{j,k\}$ then $\ell'<p$ and $q<\ell$.
By Lemma~\ref{lem.linearalg} there exists $b\in B$ such that $0\neq \c_{\ell'\ell}(b\cdot x)=\c_{w_{\ell'}w_\ell}(\dot wb\cdot x)$.
Since $b^{-1}\dot w^{-1}B\in C_{w^{-1}}\subset \B(x,H)$, we must have $E_{w_{\ell'}w_\ell}\in H$.
Now by Lemma~\ref{lem.elsh}, since $v_p\le w_{\ell'} $ and $v_q\ge w_\ell$, we have $E_{v_pv_q}\in H$, as desired.
Next, we consider the case in which $p\in\{j,k\}$ and $q\in\{j',k'\}$.
Since $i<p$ and $q<i'$, by Lemma~\ref{lem.linearalg} there exists $b\in B$ such that $0\neq \c_{ii'}(b\cdot x)=\c_{w_{i}w_{i'}}(\dot wb\cdot x)$.
Since $b^{-1}\dot w^{-1}B\in C_{w^{-1}}\subset \B(x,H)$, then $E_{w_{i}w_{i'}}\in H$.
Now by Lemma~\ref{lem.elsh}, by $v_p\le w_{i} $ and $v_q\ge w_{i'}$ we have $E_{v_pv_q}\in H$, as desired.  This concludes the proof.
\end{proof}

%\newpage

%\nocite{*}
\bibliographystyle{alpha}
\bibliography{ref.bib}

\end{document}